\documentclass[11pt]{article}
\pdfoutput=1
\usepackage{amsbsy,amsmath,amsthm,amssymb,graphicx}
\usepackage{verbatim}
\usepackage{placeins, float, subfig}
\usepackage{multirow}

\global\long\def\inv#1{\frac{1}{#1}}

\newcommand{\sX}{\mathsf{X}}

\newtheorem{theorem}{Theorem}
\newtheorem{proposition}{Proposition}
\newtheorem{corollary}{Corollary}
\newtheorem{lemma}{Lemma}

\theoremstyle{remark}

\newtheorem{example}{Example}

\usepackage[sort,longnamesfirst]{natbib}
\newcommand{\pcite}[1]{\citeauthor{#1}'s \citeyearpar{#1}}
\usepackage{geometry} 
\DeclareMathOperator{\var}{Var}

\newcommand{\cF}{\mathcal{F}}
\newcommand{\cI}{\mathcal{I}}

\newcommand{\X}{\mathsf{X}}

\newcommand {\R} {\mathbb{R}}
\newcommand {\F} {\mathbb{F}}
\newcommand {\G} {\mathbb{G}}

\newcommand {\N} {\mathbb{N}}

\def\baro{\vskip  .2truecm\hfill \hrule height.5pt \vskip  .2truecm}
\def\barba{\vskip -.1truecm\hfill \hrule height.5pt \vskip .4truecm}
\begin{document}

\title{Markov Chain Monte Carlo Estimation of Quantiles}

\author{ Charles Doss \\ School of Statistics \\ University of
Minnesota \\ {\tt cdoss@umn.edu} \and James M. Flegal\footnote{Research
  supported by the National Science Foundation.} \\Department of
Statistics \\ University of California, Riverside \\ {\tt
  jflegal@ucr.edu} \and Galin L. Jones
  \footnote{Research supported by the National
    Institutes of Health and the National Science Foundation.}\\
  School of Statistics \\
  University of Minnesota \\
  {\tt galin@umn.edu}  \and  Ronald C. Neath \\
  Department of Mathematics and Statistics \\
  Hunter College, City University of New York \\
  {\tt rneath@hunter.cuny.edu} }
\maketitle

\maketitle

\begin{abstract}
  We consider quantile estimation using Markov chain Monte Carlo and
  establish conditions under which the sampling distribution of the
  Monte Carlo error is approximately Normal.  Further, we investigate
  techniques to estimate the associated asymptotic variance, which
  enables construction of an asymptotically valid interval estimator.
  Finally, we explore the finite sample properties of these methods
  through examples and provide some recommendations to practitioners.  
\end{abstract}

\section{Introduction}
\label{sec:Intro}

Let $\pi$ denote a probability distribution having support $\X
\subseteq \R^{d}$, $d \ge 1$.  If $W \sim \pi$ and $g : \X \to \R$ is
measurable, set $V=g(W)$.  We consider estimation of quantiles of the
distribution of $V$.  Specifically, if $0 < q < 1$ and $F_{V}$ denotes
the distribution function of $V$, then our goal is to obtain
\begin{equation*}
\xi_{q} := F_{V}^{-1} (q) = \inf \{ v: F_{V}(v) \ge q \} \; .
\end{equation*}
We will assume throughout that $F_{V}(x)$ is absolutely continuous and
has continuous density function $f_{V}(x)$ such that $0 <
f_{V}(\xi_{q}) < \infty$. Notice that this means $\xi_{q}$ is the
unique solution $y$ of $F_{V} (y -) \le q \le F_{V} (y)$.

Typically, it is not possible to calculate $\xi_{q}$ directly.  For
example, a common goal in Bayesian inference is calculating the
quantile of a marginal posterior distribution.  In these settings, the
quantile estimate is typically based upon Markov chain Monte Carlo
(MCMC) simulation methods and is almost always reported without
including any notion of the simulation
error. \citet{raft:lewi:how:1992} consider quantile estimation using
MCMC, but their method is based on approximating the MCMC process with
a two-state Markov chain, and does not produce an estimate of the
simulation error; see also \citet{broo:robe:1999} and
\citet{cowl:carl:1996} who study the properties of the method proposed
by \citet{raft:lewi:how:1992}. In contrast, our work enables
practitioners to rigorously asses the simulation error, and hence
increase the reliability of their inferences.

 The basic MCMC method entails simulating a Markov chain $X = \{
 X_{0}, X_{1}, \ldots \}$ having invariant distribution $\pi$.  Define
 $Y = \{Y_{0}, Y_{1}, \ldots \} =\{ g(X_{0}), g(X_{1}), \ldots \}$.
 If we observe a realization of $X$ of length $n$ and let $Y_{n(j)}$
 denote the $j$th order statistic of $\{Y_{0}, \ldots, Y_{n-1}\}$,
 then we estimate $\xi_{q}$ with
\begin{equation} 
\label{eq:estimator} 
\hat{\xi}_{n,q} : = Y_{n(j)} \quad \text{ where } ~~ j-1 < nq \le j \; . 
\end{equation}
We will see that $\hat{\xi}_{n,q}$ is strongly consistent for
$\xi_{q}$.  While this justifies the use of $\hat{\xi}_{n,q}$, it will
be more valuable if we can also assess the unknown Monte Carlo error,
$\hat{\xi}_{n,q} - \xi_{q}$. We address this in two ways.  The first
is by finding a  function $b : \mathbb{N} \times (0, \infty) \to
[0, \infty)$ such that for all $\epsilon > 0$
\begin{equation}
\label{eq:prob bd}
\Pr \left( |\hat{\xi}_{n,q} - \xi_{q} | > \epsilon \right) \le b(n,
\epsilon) \; .
\end{equation}
We also assess the Monte Carlo error through its approximate sampling
distribution.  We will show that under a mixing condition on $X$, a
quantile central limit theorem (CLT) will obtain; this mixing
condition is much weaker than the mixing conditions required for a CLT
for a sample mean \citep{jone:2004}. For now, assume there exists a
constant $\gamma^{2}(\xi_{q})> 0$ such that as $n \to \infty$
\begin{equation}
\label{eq:intro clt}
\sqrt{n}(\hat{\xi}_{n,q} - \xi_{q}) \stackrel{d}{\to} \text{N}(0,
\gamma^{2}(\xi_{q})) \; .  
\end{equation}
Note that $\gamma^{2}(\xi_{q})$ must account for the serial dependence
present in a non-trivial Markov chain and hence is more difficult to
estimate well than when $X$ is a random sample.  However, if we can
estimate $\gamma^{2}(\xi_{q})$ with, say $\hat{\gamma}_{n}^{2}$, then
an interval estimator of $\xi_{q}$ is
\[ 
\hat{\xi}_{n,q} \pm t_{*} \frac{\hat{\gamma}_{n}}{\sqrt{n}} 
\]
where $t_{*}$ is an appropriate Student's $t$ quantile. Such
intervals, or at least, the \textit{Monte Carlo standard error}
(MCSE), $\hat{\gamma}_{n} / \sqrt{n}$, are useful in assessing the
reliability of the simulation results as they explicitly describe the
level of confidence we have in the reported number of significant
figures in $\hat{\xi}_{n,q}$.  
For more on this approach see \citet{fleg:gong:2014},
\citet{fleg:hara:jone:2008}, \citet{fleg:jone:2011},
\citet{geye:2011}, \citet{jone:hara:caff:neat:2006} and
\citet{jone:hobe:2001}.

We consider three methods for implementing this recipe, all of which
produce effective interval estimators of $\xi_{q}$.  The first two are
based on the CLT at \eqref{eq:intro clt} where we consider using the
method of batch means (BM) and the subsampling bootstrap method (SBM)
to estimate $\gamma^{2}(\xi_{q})$.  Regenerative simulation (RS) is
the third method, but it requires a slightly different quantile CLT
than that in \eqref{eq:intro clt}.  Along the way we show that
significantly weaker conditions are available for the RS-based
expectation estimation case previously studied in
\citet{hobe:jone:pres:rose:2002} and \citet{mykl:tier:yu:1995}.

The remainder is organized as follows.  We begin in
Section~\ref{sec:background} with a brief introduction to some
required Markov chain theory.  In Section~\ref{sec:BM} we consider
estimation of $\xi_{q}$ with $\hat{\xi}_{n,q}$, establish a CLT for
the Monte Carlo error, and consider how to obtain MCSEs using BM and
SBM.  In Section~\ref{sec:RS}, we consider RS, establish an
alternative CLT and show how an MCSE can be obtained.  In
Section~\ref{sec:Applications}, we illustrate the use of the methods
presented here and investigate their finite-sample properties in three
examples.  Finally, in Section~\ref{sec:Discussion} we summarize our
results and conclude with some practical recommendations.

\section{Markov chain background}  \label{sec:background}

In this section we give some essential preliminary material.  Recall
that $\pi$ has support $\sX$ and let $\mathcal{B}(\X)$ be the Borel
$\sigma$-algebra.  For $n \in \N = \{1, 2, 3,\ldots \}$, let the
$n$-step Markov kernel associated with $X$ be $P^{n}(x,dy)$.  Then if
$A \in \mathcal{B}(\X)$ and $k\in \{0, 1, 2, \ldots \}$, $P^{n}(x,A) =
\Pr (X_{k+n} \in A | X_{k} =x) $.  Throughout we assume $X$ is Harris
ergodic ($\pi$-irreducible, aperiodic, and positive Harris
recurrent--see \citet{meyn:twee:2009} for definitions) and has
invariant distribution $\pi$.

Let $\|\cdot\|$ denote the total variation norm.  Further, let $M: \sX
\mapsto \R^+$ with $E_{\pi} M < \infty$ and $\psi: \N \mapsto \R^+$ be
decreasing such that
\begin{equation} \label{eq:tvn}
\Vert P^n (x, \cdot) - \pi(\cdot) \Vert \leq M(x) \psi(n) \; .
\end{equation}
\textit{Polynomial ergodicity of order $m$} where $m > 0$ means
\eqref{eq:tvn} holds with $\psi(n) = n^{-m}$.  \textit{Geometric
  ergodicity} means \eqref{eq:tvn} holds with $\psi(n) = t^n$ for some
$0 < t < 1$.  \textit{Uniform ergodicity} means that $X$ is
geometrically ergodic and $M$ is bounded. 

An equivalent characterization of uniform ergodicity is often more
convenient for applications. The Markov chain $X$ is uniformly ergodic
if and only if there exists a probability measure $\phi$ on $\sX$,
$\lambda > 0$, and an integer $n_{0} \ge 1$ such that
\begin{equation} \label{eq:unif}
P^{n_{0}} ( x , \cdot ) \ge \lambda \phi (\cdot) \text{ for each }x \in \sX \; . 
\end{equation}
When \eqref{eq:unif} holds we have that \cite[][p. 392]{meyn:twee:2009}
\begin{equation}
\label{eq:unif tv bound}
\| P^n (x, \cdot) - \pi(\cdot) \| \leq (1-\lambda)^{\lfloor n / n_{0} \rfloor} \; .
\end{equation}
 See \citet{jone:hobe:2001} for an accessible introduction to methods
 for establishing \eqref{eq:unif} and further discussion of the
 methods for establishing \eqref{eq:tvn}.

\section{Quantile estimation for Markov chains}
\label{sec:BM}
Recall $ Y = \{Y_{0}, Y_{1}, \ldots \} =\{ g(X_{0}), g(X_{1}), \ldots
\}$ and set $F_n(y) = n^{-1} \sum_{i=0}^{n-1} I(Y_{i} \le y)$.  By the
Markov chain version of the strong law of large numbers \citep[see
  e.g.][]{meyn:twee:2009} for each $y$, $F_{n}(y) \to F_{V}(y)$ with
probability 1 as $n \to \infty$.  Using this, the proof of the
following result is similar to the proof for when $Y$ is composed of
independent and identically distributed random variables \citep[see
  e.g.][]{serf:1981} and hence is omitted.

\begin{theorem} \label{thm:xi consistent}  With
  probability 1, $\hat{\xi}_{n,q} \to \xi_{q}$ as $n \to \infty$.
\end{theorem}

While this result justifies the use of $\hat{\xi}_{n, q}$ as an
estimator of $\xi_{q}$, it does not allow one to assess the unknown
Monte Carlo error $\hat{\xi}_{n,q} - \xi_{q}$ for any finite $n$. In
Section~\ref{sec:finite sample} we establish conditions under which
\eqref{eq:prob bd} holds, while in Section~\ref{sec:clt} we examine
the approximate sampling distribution of the Monte Carlo error.

\subsection{Monte Carlo error under stationarity}
\label{sec:finite sample}

We will consider (in this subsection only) a best-case scenario where
$X_{0} \sim \pi$, that is, the Markov chain $X$ is stationary.  We
begin with a refinement of a result due to \citet{wang:hu:yang:2011}
to obtain a useful description of how the Monte Carlo error decreases
with simulation sample size and the convergence rate of the Markov
chain. The proof is given in Appendix~\ref{app:wangetal}.

\begin{proposition} \label{prop:wangetal} Suppose the Markov chain $X$
  is polynomially ergodic of order $m > 1$.  If $\delta \in
  (9/(10+8m), \, 1/2)$, then, with probability 1, for sufficiently
  large $n$, there is a positive constant $C_{0}$ such that
  $\hat{\xi}_{n,q} \in [\,\xi_{q} - C_{0}n^{-1/2 + \delta} \sqrt{\log
      n}, \, \xi_{q} + C_{0}n^{-1/2 + \delta} \sqrt{\log n}\,]$.
\end{proposition}

For the rest of this section we consider finite sample properties of
the Monte Carlo error in the sense that our goal is to find an
explicit function $b : \mathbb{N} \times (0,\infty) \to [0, \infty)$
such that \eqref{eq:prob bd} holds.  There has been some research on
this in the context of estimating expectations using MCMC
\citep[e.g.][]{latu:niem:2011, latu:niem:2012, rudo:2012}, but this
has not been considered in the quantile case.  The proofs of the
remaining results in this section can be found in
Appendix~\ref{app:proof complete}.

\begin{theorem} \label{thm:complete bosq} If $X$ satisfies
  \eqref{eq:tvn}, then for any integer $a \in \left[ 1, n/2 \right]$
  and any $\epsilon > 0$ and $0<\delta<1$
\[
\Pr \left( | \hat{\xi}_{n,q} - \xi_{q} | > \epsilon \right) \le 8 \exp \left\{ - \frac{a \gamma^2}{8}   \right\} + 22a \left( 1 + \frac{4}{\gamma} \right) ^{1/2} \psi\left( \left\lfloor \frac{n}{2a} \right\rfloor \right) E_{\pi} M \; ,
\]
where $\gamma=\gamma(\delta,\epsilon) = \min \left\{ F_{V} ( \xi_{q}
  + \epsilon) - q, \delta (q - F_{V} ( \xi_{q} - \epsilon)) \right\} $. 
\end{theorem}

To be useful Theorem~\ref{thm:complete bosq} requires bounding
$\psi(n) E_{\pi}M$.  There has been a substantial amount of work in
this area \citep[see e.g.][]{baxe:2005, fort:moul:2003, rose:1995a},
but these methods have been applied in only a few practically relevant
settings \citep[see e.g.][]{jone:hobe:2001, jone:hobe:2004}.  However,
in the uniformly ergodic case we have the following easy corollary.

\begin{corollary} \label{cor:unif bounds} If $X$ satisfies
  \eqref{eq:unif}, then we have for any $a \in [1, n/2]$, any
  $\epsilon > 0$ and any $0<\delta<1$
\[
\Pr \left( | \hat{\xi}_{n,q} - \xi_{q} | > \epsilon \right) \le 8 \exp \left\{ - \frac{a \gamma^2}{8}   \right\} + 22a \left( 1 + \frac{4}{\gamma} \right) ^{1/2}  (1-\lambda)^{\lfloor n / 2an_{0}
 \rfloor}\; ,
\]
where $\gamma=\gamma(\delta,\epsilon) = \min \left\{ F_{V} ( \xi_{q} +
  \epsilon) - q ,  \delta(q - F_{V} ( \xi_{q} - \epsilon) )\right\} $. 
\end{corollary}

\begin{example} \label{ex:linchpin}
Let 
\begin{equation}
\label{eq:da}
\pi(x,y) = \frac{4}{\sqrt{2\pi}} y^{3/2} \exp \left\{-y \left( \frac{x^{2}}{2} + 2 \right) \right\} I(0< y < \infty) 
\; .
\end{equation}
Then $Y|X=x \sim \text{Gamma}(5/2 , 2 + x^{2} /2)$ and marginally $X
\sim t(4)$--Student's $t$ with 4 degrees of freedom. Consider a
linchpin variable sampler \citep{acost:hube:jone:2014} which first
updates $X$ with a Metropolis-Hastings independence sampler having the
marginal of $X$ as the invariant distribution using a $t(3)$ proposal
distribution, then updates $Y$ with a draw from the conditional of $Y
| X$.  Letting $P$ denote the Markov kernel for this algorithm we show
in Appendix~\ref{app:linchpin} that for any measurable set $A$
\[
P((x,y), A) \ge \frac{\sqrt{9375}}{32 \pi}\int_{A} \pi(x',y')\, dx' dy'  
\]
and hence the Markov chain satisfies \eqref{eq:unif} with $n_{0}=1$ and $\lambda=\sqrt{9375}/ 32 \pi$.

Set $\delta=.99999$, $a=n/16$ and consider estimating the median of
the marginal of $X$, i.e. $t(4)$.  Then $q=1/2$ and $\xi_{1/2}=0$ so
that $\gamma =0.037422$.  Suppose we want to find the Monte Carlo
sample size required to ensure that the probability
$\hat{\xi}_{n,1/2}$ is within $.10$ of the truth is approximately 0.9.
Then Corollary~\ref{cor:unif bounds} gives
\[
\Pr \left( | \hat{\xi}_{4 \times 10^5,1/2} - \xi_{1/2} | > .1 \right) \le 0.101 \; .
\]
\end{example}

We can improve upon the conclusion of Corollary~\ref{cor:unif
  bounds}.

\begin{theorem} \label{thm:complete} If $X$ satisfies \eqref{eq:unif},
  then for every $\epsilon > 0$ and $0 < \delta < 1$
\[
\Pr \left( | \hat{\xi}_{n,q} - \xi_{q} | > \epsilon \right) \le 2 \exp
\left\{ - \frac{\lambda^2 ( n\gamma - 2 n_{0} / \lambda)^2}{2 n
  n_{0}^2} \right\} \; ,
\]
for $n > 2 n_{0} / (\lambda \gamma)$ where $\gamma = \min \left\{ F_{V} ( \xi_{q} + \epsilon) - q , \delta(q - F_{V} ( \xi_{q} - \epsilon)) \right\} $. 
\end{theorem}

\begin{example}[Continuation of Example~\ref{ex:linchpin}]
Theorem~\ref{thm:complete} yields that
\begin{equation}
\label{eq:bd2}
  \Pr \left( | \hat{\xi}_{4700,1/2} - \xi_{1/2} | > .1 \right) \le 0.101
\end{equation}
which clearly shows that the bound given in Example~\ref{ex:linchpin} is
conservative.

We will compare the bound in \eqref{eq:bd2} to the results of a
simulation experiment.  We performed 500 independent replications of
this MCMC sampler for each of 3 simulation lengths and recorded the
number of estimated medians for each that were more than .1 in
absolute value away from the median of a $t(4)$ distribution.  The
results are presented in Table~\ref{tab:linchpin} and
Figure~\ref{fig:linchpin}.  The results in Table~\ref{tab:linchpin}
show that the estimated probability in \eqref{eq:bd2} is somewhat
conservative.  On the other hand, from Figure~\ref{fig:linchpin} it is
clear that the estimation procedure is not all that stable until
$n=4700$.

\begin{table}
\begin{center}
\begin{tabular}{lccc} 
  \hline\hline
  Length & 500 & 1000 & 4700 \\
  Count  & 60 & 9 & 0 \\
  $\hat{\Pr}$ & .12 & .018 & 0 \\
  \hline
  \hline
\end{tabular}
\caption{Simulation length for each of 500 independent replications,
  counts of sample medians more than .1 away from 0 in absolute value
  and, $\hat{\Pr} ( | \hat{\xi}_{n,1/2} - \xi_{1/2} | >
    .1 )$.}
\label{tab:linchpin}
\end{center}
\end{table} 

\begin{figure} 
\begin{center}
\includegraphics[height=3.5in]{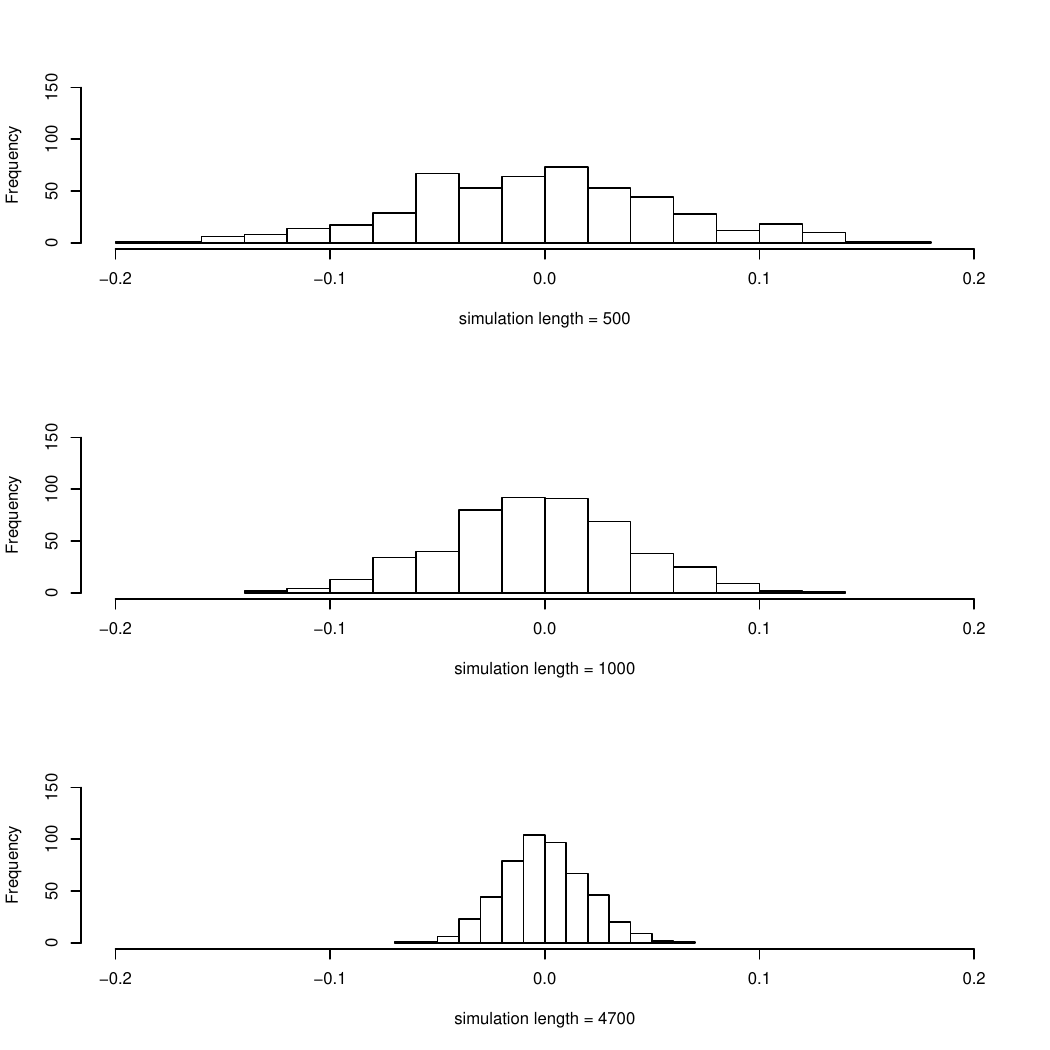}  
\caption{Histograms of 500 sample medians for each of 3 simulation lengths.}
\label{fig:linchpin} 
\end{center}
\end{figure}

\end{example}

\subsection{Central limit theorem}
\label{sec:clt}
We consider the asymptotic distribution of the Monte Carlo error
$\hat{\xi}_{n,q} - \xi_{q}$. Let
\begin{equation}
\label{eq:variance}
\sigma^{2}(y) := \text{Var}_{\pi} I(Y_{0} \le y) + 2 \sum_{k=1}^{\infty} \text{Cov}_{\pi} \left[ I(Y_{0} \le y), I(Y_{k} \le y) \right] \; . 
\end{equation}
The proof of the following result is in Appendix~\ref{app:clt proof}.

\begin{theorem} \label{thm:clt} If $X$ is polynomially ergodic of
  order $m > 1$ and if $\sigma^{2} (\xi_{q}) > 0$, then as $n \to
  \infty$
\begin{equation}
\label{eq:CLT}
\sqrt{n}(\hat{\xi}_{n,q} - \xi_{q}) \stackrel{d}{\to} \mathrm{N}(0,
\sigma^{2} (\xi_{q}) / [f_{V}(\xi_{q})]^{2}) \; .
\end{equation}
\end{theorem}

To obtain an MCSE we need to estimate $\gamma^{2}(\xi_{q})
:=\sigma^{2} (\xi_{q}) / [f_{V}(\xi_{q})]^{2}$.  We consider two
methods for doing this--in Section~\ref{sec:bm} we consider the method
of batch means while in Section~\ref{sec:Other} we consider
subsampling.

\subsubsection{Batch Means} \label{sec:bm}

To estimate $\gamma^{2}(\xi_{q})$, we substitute $\hat{\xi}_{n,q}$ for
$\xi_{q}$ and  estimate $f_{V}(\hat{\xi}_{n,q})$ and
$\sigma^{2} (\hat{\xi}_{n,q})$.

Consider estimating $f_{V}(\hat{\xi}_{n,q})$.  Consistently estimating
a density at a point has been studied extensively in the context of
stationary time-series analysis \citep[see e.g.][]{robi:1983} and many
existing results are applicable since the Markov chains in MCMC are
special cases of strong mixing processes.  In our examples we use
kernel density estimators with a Gaussian kernel to obtain
$\hat{f}_{V}(\hat{\xi}_{n,q})$, an estimator of
$f_{V}(\hat{\xi}_{n,q})$.

The quantity $\sigma^{2}(y)$, $y \in \mathbb{R}$ is familiar.  Notice that
\[
\sqrt{n} (F_n(y)   - E_{\pi} I(Y \le y)) \stackrel{d}{\to}
\mathrm{N}(0, \sigma^{2}(y))~\text{ as } n \to \infty 
\]
by the usual Markov chain CLT for sample means \citep{jone:2004}.
Moreover, we show in Corollary~\ref{cor:cont} that $\sigma^{2}(y)$ is
continuous at $\xi_{q}$.  In this context, estimating $\sigma^{2}(y)$
consistently is a well-studied problem and there are an array of
methods for doing so; see \citet{fleg:hara:jone:2008},
\citet{fleg:jone:2010}, \citet{fleg:jone:2011} and
\citet{jone:hara:caff:neat:2006}.  Here we focus on the method of
batch means for estimating $\sigma^{2}(\hat{\xi}_{n,q})$. For BM the
output is split into batches of equal size.  Suppose we obtain $n=a_n
b_n$ iterations $\{X_{0}, \ldots, X_{n-1}\}$ and for
$k=0,\ldots,a_n-1$ define $\bar{U}_{k}(\hat{\xi}_{n,q}) = b_n^{-1}
\sum_{i=0}^{b_n-1} I(Y_{k b_n+i} \le \hat{\xi}_{n,q})$.  Then the BM
estimator of $\sigma^{2}(\hat{\xi}_{n,q})$ is
\begin{equation}
\label{eq:bmvar}
\hat{\sigma}_{BM}^{2} (\hat{\xi}_{n,q}) = \frac{b_n}{a_n-1} \sum_{k=0}^{a_n-1} \left( \bar{U}_{k}(\hat{\xi}_{n,q}) - F_n(\hat{\xi}_{n,q}) \right)^{2}  \; . 
\end{equation}

Putting these two pieces together we estimate $\gamma^{2}(\xi_{q})$ with
\[
\hat{\gamma}^{2}(\hat{\xi}_{n,q}) := \frac {\hat{\sigma}_{BM}^{2}
  (\hat{\xi}_{n,q})}{[\hat{f}_{V}(\hat{\xi}_{n,q})]^2}
\]
and we can obtain an approximate $100(1 - \alpha)\%$ confidence interval for $\xi_{q}$ by 
\begin{equation} \label{eq:ci bm} 
\hat{\xi}_{n,q} \pm z_{\alpha/2} \frac{ \hat{\gamma}(\hat{\xi}_{n,q}) }{ \sqrt{n}} \; ,
\end{equation}
where $z_{\alpha/2}$ is a standard Normal quantile.

\subsubsection{Subsampling} 
\label{sec:Other} 

It is natural to consider the utility of bootstrap methods for
estimating quantiles and the Monte Carlo error.  Indeed, there has
been a substantial amount of work on using bootstrap methods for
stationary time-series \citep[e.g.][]{bert:clem:2006, buhl:2002,
  carl:1986b, datt:mcco:1993, poli:2003}.  However, in our
experience, MCMC simulations are typically sufficiently long so that
standard bootstrap methods are prohibitively computationally
expensive.

We focus on the subsampling bootstrap method (SBM) described in
general by \cite{poli:roma:wolf:1999} and, in the context of MCMC, by
\cite{fleg:2012} and \cite{fleg:jone:2011}.  The basic idea is to
split $X$ into $n-b+1$ overlapping blocks of length $b$.  We then
estimate $\xi_{q}$ over each block resulting in $n-b+1$ estimates.  To
this end, consider the $i$th subsample of $Y$, $\{Y_{i-1}, \dots, Y_{i
  + b - 2}\}$.  Define the corresponding ordered subsample as
$\{Y^{i*}_{b(1)} , \dots, Y^{i*}_{b(b)}\}$ and quantile estimator as
\[
\xi^{*}_{i} = Y^{i*}_{b(j)} \text{ where } j-1 < bq \le j \text{ for } i = 1, \dots, n - b + 1 \; .
\]
If 
\[
\bar{\xi}^* = \frac{1}{n-b+1} \sum_{i = 1}^{n-b+1}  \xi^{*}_{i} \; ,
\]
then the SBM estimator of $\gamma^{2}(\xi_{q})$ is given by 
\begin{equation*}
\hat{\gamma}^2_{S} = \frac{b}{n-b+1} \sum_{i = 1}^{n-b+1} (\xi^{*}_{i}
- \bar{\xi}^* )^2 \; .
\end{equation*}
Note that SBM avoids having to estimate the density
$f_{V}(\hat{\xi}_{n, q})$.  An approximate $100(1 - \alpha)\%$
confidence interval for $\xi_{q}$ is given by
\begin{equation} \label{eq:ci sbm} 
\hat{\xi}_{n,q} \pm z_{\alpha/2} \frac{ \hat{\gamma_{S}}(\hat{\xi}_{n,q}) }{ \sqrt{n}} \; ,
\end{equation}
where $z_{\alpha/2}$ is an appropriate standard Normal quantile.

\section{Quantile estimation for regenerative Markov chains}
\label{sec:RS}

Regenerative simulation (RS) provides an alternative estimation method
for Markov chain simulations.  RS is based on simulating an augmented
Markov chain and Theorem~\ref{thm:clt} will not apply.  We derive
an alternative CLT based on RS and consider a natural estimator of the
variance in the asymptotic Normal distribution.

Recall $X$ has $n$-step Markov kernel $P^{n}(x,dy)$ and suppose there
exists a function $s : \sX \rightarrow [0,1]$ with $E_{\pi}s > 0$ and
a probability measure $Q$ such that
\begin{equation}
\label{eq:minor.gen}
P(x,A) \geq s(x) Q(A) ~~ \mathrm{for}~\mathrm{all}~ x \in \mathsf{X} ~
\mathrm{and} ~ A \in \mathcal{B} \; . 
\end{equation}
We call $s$ the \textit{small function} and $Q$ the \textit{small
  measure}.  Define the \textit{residual measure}
\begin{equation}
\label{eq:residual}
{\mathsf R}(x, dy) = \left\{ \begin{array}{cc}
\frac{\displaystyle P(x, dy) - s(x)Q(dy)}
     {\displaystyle 1 - s(x)} & ~~s(x) < 1 \; \; \\
Q (dy) & ~~s(x) = 1   \\
\end{array} \right.
\end{equation}
so that
\begin{equation}
\label{eq:split}
P(x, dy) = s(x) Q(dy) + (1 - s(x)) {\mathsf R}(x, dy) \; .
\end{equation}
We now have the ingredients for constructing the \textit{split chain},
\[
X' = \left\{ (X_{0},\delta_{0}), (X_{1}, \delta_{1}), (X_{2},
  \delta_{2}), \ldots\right\}
\]
which lives on $\sX \times \{0,1\}$. Given $X_{i}= x$, then
$\delta_{i}$ and $X_{i+1}$ are found by \baro
\begin{enumerate}
\item Simulate $\delta_{i} \sim \mathrm{Bernoulli}( s(x) )$ 
\item If $\delta_{i}= 1$, simulate $X_{i+1}\sim Q(\cdot)$; otherwise 
$X_{i+1}\sim {\mathsf R}(x, \cdot)$.  
\end{enumerate}
\barba

Two things are apparent from this construction.  First, by
\eqref{eq:split} the marginal sequence $\left\{ X_{n} \right\}$ has Markov
transition kernel given by $P$.  Second, the set of $n$ for which
$\delta_{n-1} = 1$, called \emph{regeneration times}, represent
times at which the chain probabilistically restarts itself in the
sense that $X_{n} \sim Q(\cdot)$ does not depend on $X_{n-1}$.

The main practical impediment to the use of regenerative simulation
would appear to be the means to simulate from the residual kernel
${\mathsf R}(\cdot, \cdot)$, defined at \eqref{eq:residual}.  Interestingly, as
shown by \citet{mykl:tier:yu:1995}, this is essentially a non-issue,
as there is an equivalent update rule for the split chain which does not
depend on ${\mathsf R}$.  Given $X_{k} = x$, find $X_{k+1}$ and
$\delta_{k}$ by

\baro
\begin{enumerate}
\item Simulate $X_{k+1} \sim P(x, \cdot)$
\item Simulate $\delta_{k} \sim \text{Bernoulli}(r(X_{k}, X_{k+1}))$ where
\begin{equation*}
r(x,y) = \frac{s(x) Q(dy) }{P(x, dy)} \; .
\end{equation*}
\end{enumerate}
\barba

RS has received considerable attention in the case where
either a Gibbs sampler or a full-dimensional Metropolis-Hastings
sampler is employed. In particular, \citet{mykl:tier:yu:1995} give
recipes for establishing minorization conditions as in
\eqref{eq:minor.gen}, which have been implemented in
several practically relevant statistical models; see e.g.
\citet{doss:tan:2013, gilk:robe:sahu:1998, hobe:jone:robe:2006,
  jone:hara:caff:neat:2006, jone:hobe:2001, roy:hobe:2007}.  

Suppose we start $X'$ with $X_{0}\sim Q$; one can always discard the
draws preceding the first regeneration to guarantee this, but it is
frequently easy to draw directly from $Q$
\citep{hobe:jone:pres:rose:2002, mykl:tier:yu:1995}. We will write
$E_{Q}$ to denote expectation when the split chain is started with
$X_{0} \sim Q$.  Let $0 = \tau_0 < \tau_1 < \tau_2 < \ldots$ be the
regeneration times so that $\tau_{t+1} = \min \left\{ i > \tau_{t}:
  \delta_{i-1} = 1 \right\}$.  Assume $X'$ is run for $R$ tours so
that the simulation is terminated the $R$th time that a
$\delta_{i}=1$.  Let $\tau_{R}$ be the total length of the simulation
and $N_{t} = \tau_{t} - \tau_{t-1}$ be the length of the $t$th tour.
Let $h : \sX \to \mathbb{R}$, $V_{i} = h(X_{i})$ and define
\[
S_{t} = \sum_{i=\tau_{t-1}}^{\tau_{t} - 1} V_{i} \quad \text{ for } t
= 1, \ldots, R \; .
\]
The split chain construction ensures that the pairs $(N_t, S_t)$ are
independent and identically distributed.  It is straightforward
to show
\citep{hobe:jone:pres:rose:2002,meyn:twee:2009,mykl:tier:yu:1995} that
if $E_{Q} N_t^2 < \infty$ and $E_{Q} S_t^2 < \infty$, then as $R \to
\infty$,
\begin{equation}
\label{eq:RS SLLN}
\overline{h}_{\tau_{R}}= \frac{\sum_{t=1}^{R} S_{t}}{\sum_{t=1}^{R}
  N_{t}} = \frac{\overline{S}}{\overline{N}} \to E_{\pi} h \qquad
\text{with probability 1}
\end{equation}
and, if $\Gamma =  E_{Q} \left[ (S_{1} - N_{1} E_{\pi} h)^{2}\right] / \left[
  E_{Q}(N_{1}) \right]^{2}$, then
\begin{equation}
\label{eq:RS CLT}
\sqrt{R}( \overline{h}_{\tau_{R}} - E_{\pi} h) \stackrel{d}{\to}
\text{N} (0, \Gamma) \; .
\end{equation}
Moreover, there is an easily calculated consistent estimator of
$\Gamma$; see \citet{hobe:jone:pres:rose:2002}. However, the required
moment conditions, $E_{Q} N_t^2 < \infty$ and $E_{Q} S_t^2 < \infty$,
are difficult to check in practice.  \citet{hobe:jone:pres:rose:2002}
showed that these moment conditions will hold if the Markov chain $X$
is geometrically ergodic and there exists $\delta > 0$ such that
$E_{\pi}|h|^{2+\delta} < \infty$.  Our next result significantly
weakens the required mixing conditions.  The proof can be found in
Appendix~\ref{app:moments}.

\begin{theorem} \label{thm:moments} If $X$ is polynomially ergodic of
  order $m >1$ and there exists $\delta > 2 /(m-1)$ such that
  $E_{\pi} |h|^{2+\delta} < \infty$, then $E_{Q} N_{t}^2 < \infty$ and
  $E_{Q} S_t^2 < \infty$.
\end{theorem}

In the sequel we use Theorem~\ref{thm:moments} to develop an RS-based
CLT for quantiles.

\subsection{Quantile estimation}
Recall $ Y = \{Y_{0}, Y_{1}, \ldots \} =\{ g(X_{0}), g(X_{1}), \ldots
\}$ and define
\[
S_{t} (y) = \sum_{i=\tau_{t-1}}^{\tau_{t} - 1} I(Y_{i} \le y) \quad \text{ for } t = 1, \ldots, R \; .
\]
Note that $0 \le S_{t} (y) \le N_{t}$ for all $y \in \R$, and hence
$E_{Q} (S_t (y))^2 \le E_{Q} (N_t)^2$.  For each $y \in \mathbb{R}$ set
\[
\Gamma (y) = E_{Q} \left[ \left( S_{1} (y) - F_{V}(y) N_{1} \right)^2 \right] / \left[ E_{Q} (N_{1}) \right]^2 \, ,
\]
which exists under the conditions of Theorem~\ref{thm:moments}.

Let $j = \tau_{R} q + o(\sqrt{\tau_R})$ as $R \to \infty$ and consider
estimating $\xi_{q}$ with $Y_{\tau_{R} (j)}$, that is, the $j$th order
statistic of $Y_{1}, \ldots, Y_{\tau_{R}}$. The proof of the following
CLT is given in Appendix~\ref{app:RS CLT}.

\begin{theorem} \label{thm:clt regen} If $X$ is polynomially
  ergodic of order $m > 1$, then, as $R \to \infty$,
\begin{equation*}
  \sqrt{R} \left( Y_{\tau_{R} (j)} - \xi_{q} \right) \stackrel{d}{\rightarrow} \mathrm{N} \left( 0, \Gamma \left( \xi_{q} \right) / f_{V}^2 \left( \xi_{q} \right) \right) \; .
\end{equation*}
\end{theorem}

  Since $\hat{\xi}_{\tau_{R},q}$ requires $j$ such that $0 \le j -
  \tau_{R} q < 1$ we have the following corollary.

\begin{corollary} \label{cor:RS} If $X$ is polynomially ergodic of
  order $m > 1$, then, as $R \to \infty$,
\begin{equation*}
  \sqrt{R} ( \hat{\xi}_{\tau_{R},q} - \xi_{q} ) \stackrel{d}{\rightarrow} \mathrm{N} \left( 0, \Gamma \left( \xi_{q} \right) / f_{V}^2 \left( \xi_{q} \right) \right) \; .
\end{equation*}
\end{corollary}

To obtain an MCSE we need to estimate $\gamma_{R}^{2}(\xi_q) :=\Gamma
\left( \xi_{q} \right) / f_{V}^2 \left( \xi_{q} \right) $.  We
substitute $\hat{\xi}_{\tau_{R},q}$ for $\xi_{q}$ and separately
consider $\Gamma(\hat{\xi}_{\tau_{R},q})$ and
$f_{V}(\hat{\xi}_{\tau_{R},q})$.  Of course, we can handle estimating
$f_{V}(\hat{\xi}_{\tau_{R},q})$ exactly as before, so all we need to
concern ourselves with is estimation of
$\Gamma(\hat{\xi}_{\tau_{R},q})$.

We can recognize $\Gamma(y)$ as the variance of an asymptotic Normal
distribution. Let $\hat{F}_{R} (y)=\sum_{t=1}^R S_t (y) / \sum_{t=1}^R
N_t$. Then, using \eqref{eq:RS SLLN}, we have that, with probability
1, as $R \rightarrow \infty$, $\hat{F}_{R}(y) \to F_{V}(y)$ for each
fixed $y$.  Moreover, using \eqref{eq:RS CLT}, for each $y \in \R$, as
$R \rightarrow \infty$,
\begin{equation*}
\sqrt{R} ( \hat{F}_{R} (y) - F_{V} (y) )
\stackrel{d}{\rightarrow} \text{N} \left( 0, \Gamma (y) \right) \; .
\end{equation*}
We can consistently estimate $\Gamma (y)$ for each $y$ with
\[
\hat{\Gamma}_{R}(y)  = \frac{1}{R \bar{N}^{2}} \sum_{t=1}^{R}(S_{t} (y) - \hat{F}_{R} (y) N_t )^{2} \; .
\]
Letting $\hat{f}_{V}(\hat{\xi}_{\tau_{R}, q})$ denote an estimator of
$f_{V}(\hat{\xi}_{\tau_{R}, q})$ we estimate $\gamma_{R}^{2}(\xi_{q})$
with
\[
\hat{\gamma}_{R}^{2}(\hat{\xi}_{\tau_{R}, q}) := \frac{\hat{\Gamma}(\hat{\xi}_{\tau_{R}, q})}{\hat{f}_{V}(\hat{\xi}_{\tau_{R}, q})} \; .
\]
Finally, if $t_{R-1, \alpha/2}$ is a quantile from a Student's $t$
distribution with $R-1$ degrees of freedom, a $100(1 - \alpha)\%$
confidence interval for $\xi_{q}$ is
\begin{equation} \label{eq:ci rs} \hat{\xi}_{\tau_{R},q} \pm t_{R-1,
    \alpha/2} \frac{\hat{\gamma}_{R}(\hat{\xi}_{\tau_{R},
      q})}{\sqrt{R}} \; .
\end{equation}

\section{Examples}
\label{sec:Applications}

In this section, we investigate the finite-sample performance of the
confidence intervals for $\xi_{q}$ defined at \eqref{eq:ci bm},
\eqref{eq:ci sbm}, and \eqref{eq:ci rs} corresponding to BM, SBM and
RS, respectively.  While each of our examples are quite different, the
simulation studies were conducted using a common methodology.  In each
case we perform many independent replications of the MCMC sampler.
Each replication was performed for a fixed number of regenerations,
then confidence intervals were constructed on the same MCMC
output. For the BM-based and SBM-based intervals we always used $b_{n}
= \lfloor n^{1/2} \rfloor$, which has been found to work well in other
settings \citep{jone:hara:caff:neat:2006, fleg:jone:2010, fleg:2012}.
In order to estimate coverage probabilities we require the true values
of the quantiles of interest.  These are available in only one of our
examples.  In the other example we estimate the truth with an
independent long run of the MCMC sampler.  The details are described
in the following sections.

\subsection{Polynomial target distribution}
\label{sec:t}

\cite{jarn:robe:2007} studied MCMC for heavy-tailed target
distributions. A target distribution is said to be {\em polynomial of
  order $r$} if its density satisfies $f(x) = (l(|x|) / |x|)^{1+r}$,
where $r > 0$ and $l$ is a normalized slowly varying function---a
particular example is Student's $t$-distribution. We consider
estimating quantiles of Student's $t$-distribution $t(v)$ for degrees
of freedom $v=3$, 6, and 30; the $t(v)$ distribution is polynomial of
order $v$.  We use a Metropolis random walk algorithm with jump
proposals drawn from a $\text{N}(0, \sigma^2)$ distribution. By
Proposition 3 of \cite{jarn:robe:2007}, a Metropolis random walk for a
$t(v)$ target distribution using any proposal kernel with finite
variance is polynomially ergodic of order $v/2$.  Thus the conditions
of Theorem \ref{thm:clt} and Corollary \ref{cor:RS} are satisfied for
$v > 2$.

We tuned the scale parameter $\sigma^2$ in the proposal distribution
in order to minimize autocorrelation in the resulting chain (second
row of Table \ref{tab:t_example}); the resulting acceptance rates varied
from about 25\% for $t(3)$ with $\sigma = 5.5$, the heaviest tailed
target distribution, to about 40\% for $t(30)$ with $\sigma = 2.5$.
Regeneration times were identified using the retrospective method of
\cite{mykl:tier:yu:1995}; see Appendix \ref{app:regen_details} for
implementation details, and the bottom rows of Table \ref{tab:t_example}
for regeneration performance statistics (mean and SD of tour lengths).
For each of the $10^4$ replications and using each of \eqref{eq:ci
  bm}, \eqref{eq:ci sbm}, and \eqref{eq:ci rs} we computed a 95\%
confidence interval for $\xi_q$ for $q = 0.50$, 0.75, 0.90, and 0.95.

Empirical coverage rates (percentage of the $10^4$ intervals that
indeed contain the true quantile $\xi_q$) are shown in Table
\ref{tab:MetropRW_t}.  We first note that, as might be expected,
agreement with the nominal coverage rate is closer for estimation of
the median than for the tail quantiles $\xi_{.90}$ and $\xi_{.95}$.
As for comparing the three approaches to MCSE estimation, we find that
agreement with the nominal coverage rate is closest for SBM on
average, but SBM also shows the greatest variability between cases
considered, including a couple of instances ($\xi_{.90}$ and
$\xi_{.95}$ for the $t(3)$ target distribution) where the method
appears overly conservative.  Results for BM and RS show less
variability than those of SBM, with agreement with the nominal rate
being slightly better for RS.  

Table \ref{tab:CI_width} shows the mean and standard deviation of
interval half-widths for the three cases (defined by the quantile $q$
and number of regenerations $R$) in which all empirical coverage rates
were at least 0.935.  The most striking result here is the huge
variability in the standard errors as computed by SBM, particularly
for the heaviest tailed target distribution.  Results for BM and RS
are comparable, with RS intervals being slightly wider and having
slightly less variability.  The SBM intervals are generally as wide or
wider, demonstrating again the apparent conservatism of the method.

\begin{table}
\begin{center}
\begin{tabular}{lccc} 
\hline
 & \multicolumn{3}{c}{Target distribution} \\ 
 & ~$t(30)$~ & ~$t(6)$~ & ~$t(3)$~ \\
\hline
Tuning parameter $\sigma$~~~~  & 2.5 & 3.5 & 5.5 \\ 
Mean tour length               & 3.58 & 4.21 & 5.60 \\
SD of tour lengths             & 3.14 & 3.80 & 5.23 \\
\hline
\end{tabular}
\caption{Metropolis random walk on $t(v)$ target distribution with
  $\text{N}(0, \sigma^2)$ jump proposals, example of Section~\ref{sec:t}.  
}
\label{tab:t_example} 
\end{center}
\end{table} 

\begin{table}
\begin{center}
\begin{tabular}{|cc|ccc|ccc|}
\multicolumn{8}{c}
{Estimating $\xi_q$ of $t(v)$ distribution based on Normal Metropolis RW} \\
\hline
 & & \multicolumn{3}{c|}{500 regenerations} 
  & \multicolumn{3}{c|}{2000 regenerations} \\
~Quantile~ & ~Method~ & $t(30)$ & $t(6)$ & $t(3)$ & $t(30)$ & $t(6)$ & $t(3)$ \\
\hline
             & BM  & ~~0.941~ & ~0.939~ & ~0.935~~ & ~~0.946~ & ~0.946~ & ~0.947~~ \\
$~q = 0.50~$ & SBM & ~0.946 & 0.945 & 0.947~ & ~0.948 & 0.949 & 0.950~ \\
             & RS  & ~0.952 & 0.951 & 0.946~ & ~0.951 & 0.950 & 0.952~ \\
\hline
           & BM  & ~0.935 & 0.931 & 0.932~ & ~0.946 & 0.939 & 0.945~ \\
$q = 0.75$ & SBM & ~0.944 & 0.948 & 0.955~ & ~0.948 & 0.948 & 0.961~ \\
           & RS  & ~0.947 & 0.942 & 0.942~ & ~0.951 & 0.944 & 0.951~ \\
\hline
           & BM  & ~0.923 & 0.916 & 0.916~ & ~0.941 & 0.935 & 0.933~ \\
$q = 0.90$ & SBM & ~0.926 & 0.942 & 0.957~ & ~0.948 & 0.955 & 0.976~ \\
           & RS  & ~0.933 & 0.928 & 0.927~ & ~0.945 & 0.940 & 0.940~ \\
\hline
           & BM  & ~0.906 & 0.898 & 0.895~ & ~0.934 & 0.930 & 0.931~ \\
$q = 0.95$ & SBM & ~0.888 & 0.898 & 0.932~ & ~0.935 & 0.956 & 0.972~ \\
           & RS  & ~0.914 & 0.909 & 0.906~ & ~0.938 & 0.936 & 0.935~ \\
\hline
\end{tabular}
 \caption{ Empirical coverage rates for nominal 95\% confidence
   intervals for $\xi_q$, the $q$-quantile of the $t(v)$
   distribution. Based on $n=10^4$ replications of 500 or 2000
   regenerations of a Metropolis random walk with jump proposals drawn
   from a Normal distribution. The Monte Carlo standard errors for the
   observed sample proportions fall between 1.5E-3 and 3.2E-3. }
\label{tab:MetropRW_t}
\end{center}
\end{table}

\begin{table}
\begin{center}
\begin{tabular}{c|ccc}
\multicolumn{4}{l}{$q = 0.50$, $R = 500$} \\
\hline
 & \multicolumn{3}{c}{Target distribution} \\
 ~MCSE Method~ & $~t(30)$ & $t(6)$ & $t(3)~$ \\
\hline
BM  & ~~0.120 (0.022)~  & ~0.127 (0.023)~  & ~0.134 (0.025)~~  \\
SBM & ~0.121 (0.016)  & 0.129 (0.021)  & 0.146 (0.099)~  \\
RS  & ~0.124 (0.015)  & 0.131 (0.017)  & 0.140 (0.020)~  \\
\hline
\end{tabular}
\end{center}
\medskip
\begin{center}
\begin{tabular}{c|ccc}
\multicolumn{4}{l}{$q = 0.50$, $R = 2000$} \\
\hline
 & \multicolumn{3}{c}{Target distribution} \\
 ~MCSE Method~ & $~t(30)$ & $t(6)$ & $t(3)~$ \\
\hline
BM  & ~~0.061 (0.008)~  & ~0.064 (0.008)~  & ~0.068 (0.008)~~  \\
SBM & ~0.060 (0.005)  & 0.064 (0.006)  & 0.072 (0.066)~  \\
RS  & ~0.062 (0.004)  & 0.065 (0.005)  & 0.069 (0.006)~  \\
\hline
\end{tabular}
\end{center}
\medskip
\begin{center}
\begin{tabular}{c|ccc}
\multicolumn{4}{l}{$q = 0.75$, $R = 2000$} \\
\hline
 & \multicolumn{3}{c}{Target distribution} \\
 ~MCSE Method~ & $~t(30)$ & $t(6)$ & $t(3)~$ \\
\hline
BM  & ~~0.066 (0.009)~  & ~0.072 (0.009)~  & ~0.080 (0.011)~~  \\
SBM & ~0.066 (0.006)  & 0.074 (0.012)  & 0.094 (0.095)~  \\
RS  & ~0.067 (0.005)  & 0.073 (0.006)  & 0.082 (0.008)~  \\
\hline
\end{tabular}
\medskip
\label{tab:CI_width}
\caption{{ Mean and standard deviation for half-widths of 95\% confidence 
intervals for $\xi_q$, in $10^4$ replications of Normal Metropolis random 
walk with $R$ regenerations.}}
\label{tab:CI_width}
\end{center}
\end{table}

\subsection{Probit regression} \label{sec:probit}

\citet{vand:meng:2001} report data which is concerned with the
occurrence of latent membranous lupus nephritis.  Let $y_i$ be an
indicator of the disease (1 for present), $x_{i1}$ be the difference
between IgG3 and IgG4 (immunoglobulin G), and $x_{i2}$ be IgA
(immunoglobulin A) where $i=1, \ldots,55$.  Let $\Phi$ denote the
standard normal distribution function and suppose
\[
 \text{Pr} (Y_i = 1 ) = \Phi \left( \beta_0 + \beta_1 x_{i1} + \beta_2 x_{i2} \right)
\]
and take the prior on $\beta := (\beta_0 , \beta_1, \beta_2 )$ to be
Lebesgue measure on $\mathbb{R}^{3}$.  \citet{roy:hobe:2007} show that
the posterior $\pi(\beta | y)$ is proper.  Our goal is to report a
median and an 80\% Bayesian credible region for each of the three
marginal posterior distributions.  Denote the $q$th quantile
associated with the marginal for $\beta_{j}$ as $\xi^{(j)}_{q}$ for $j
= 0,1,2$.  Then the vector of parameters to be estimated is
\[
\Xi = \left( \xi^{(0)}_{.1}, \xi^{(0)}_{.5}, \xi^{(0)}_{.9} , \xi^{(1)}_{.1} , \xi^{(1)}_{.5} , \xi^{(1)}_{.9} , \xi^{(2)}_{.1}, \xi^{(2)}_{.5}, \xi^{(2)}_{.9} \right) \; .
\]

We will sample from the posterior using the PX-DA algorithm of
\citet{liu:wu:1999}, which \citet{roy:hobe:2007} prove is
geometrically ergodic. For a full description of this algorithm in the
context of this example see \citet{fleg:jone:2010} or \citet{roy:hobe:2007}.

We now turn our attention to comparing coverage probabilities for
estimating elements of $\Xi$ based on the confidence intervals at
\eqref{eq:ci bm}, \eqref{eq:ci sbm}, and \eqref{eq:ci rs}.  We
calculated a precise estimate from a long simulation of the PX-DA
chain and declared the observed quantiles to be the truth--see
Table~\ref{tab:truth}. \citet{roy:hobe:2007} implement RS for this
example and we use their settings exactly with 25 regenerations.  This
procedure was repeated for 1000 independent replications resulting in
a mean simulation effort of 3.89E5 (2400).  The resulting coverage
probabilities can be found in Table~\ref{tab:probit}.  Notice that for
the BM and SBM intervals all the coverage probabilities are within two
MCSEs of the nominal 0.95 level.  However, for RS only 7 of the 9
investigated settings are within two MCSEs of the nominal level.  In
addition, all of the results using RS are below the nominal 0.95
level.

Table~\ref{tab:probit} gives the empirical mean and standard deviation
of the half-width of the BM-based, RS-based, and SBM-based confidence
intervals.  Notice the interval lengths are similar across the three
methods, but the RS-based interval lengths are more variable.
Further, the RS-based intervals are uniformly wider on average than
the BM-based intervals even though they have uniformly lower empirical
coverage probabilities.

\begin{table}[H]
\begin{center}
\begin{tabular}{c|ccc}
$q$ & 0.1 & 0.5 & 0.9 \\
\hline
$\beta_0$ & -5.348 (7.21E-03) & -2.692 (4.00E-03) & -1.150 (2.32E-03) \\
$\beta_1$ & 3.358 (4.79E-03) & 6.294 (7.68E-03) & 11.323 (1.34E-02) \\
$\beta_2$ & 1.649 (2.98E-03) & 3.575 (5.02E-03) & 6.884 (8.86E-03) \\
\end{tabular}
\caption{Summary for Probit regression example of calculated ``truth''.  These calculations are based on 9E6 iterations where the MCSEs are calculated using a BM procedure.}
\label{tab:truth}
\end{center}
\end{table}

\begin{table}[H]
\begin{center}
\begin{tabular}{cc|ccc|ccc}
\multicolumn{2}{c}{} & \multicolumn{3}{c}{Probability} & \multicolumn{3}{c}{Half-Width}\\
& $q$ & 0.1 & 0.5 & 0.9 & 0.1 & 0.5 & 0.9\\
\hline
\multirow{3}{6mm}{$\beta_0$} & BM & 0.956 & 0.948 & 0.945 & 0.0671 (0.007) & 0.0377 (0.004) & 0.0222 (0.002) \\
& RS & 0.942 & 0.936 & 0.934 & 0.0676 (0.015) & 0.0384 (0.008) & 0.0226 (0.005) \\
& SBM & 0.952 & 0.947 & 0.955 & 0.0650 (0.006) & 0.0375 (0.004) & 0.0232 (0.003) \\
\hline
\multirow{3}{6mm}{$\beta_1$} & BM & 0.948 & 0.943 & 0.948 & 0.0453 (0.005) & 0.0720 (0.007) & 0.1260 (0.013) \\
& RS & 0.942 & 0.936 & 0.934 & 0.0459 (0.010) & 0.0733 (0.016) & 0.1270 (0.028) \\
& SBM & 0.954 & 0.942 & 0.940 & 0.0464 (0.005) & 0.0716 (0.007) & 0.1230 (0.012) \\
\hline
\multirow{3}{6mm}{$\beta_2$} & BM & 0.949 & 0.950 & 0.950 & 0.0287 (0.003) & 0.0474 (0.005) & 0.0825 (0.009) \\
& RS & 0.938 & 0.940 & 0.937 & 0.0292 (0.006) & 0.0481 (0.010) & 0.0831 (0.018) \\
& SBM & 0.955 & 0.948 & 0.948 & 0.0297 (0.003) & 0.0470 (0.005) & 0.0801 (0.008) \\
\end{tabular}
\caption{Summary for estimated coverage probabilities and observed CI half-widths for Probit regression example.  CIs reported have 0.95 nominal level with MCSEs equal ranging from 6.5E-3 to 7.9E-3.}
\label{tab:probit}
\end{center}
\end{table}

\subsection{A hierarchical random effects model}
\label{sub:baseball}

A well known data set first analyzed by \cite{efro:morr:1975} consists
of the batting averages of 18 Major League Baseball players in their
first 45 official at bats of the 1970 season.  Let $x_i$ denote the
batting average of the $i$th player, and $y_i = \sqrt{45} \arcsin (2
x_i - 1)$, for $i = 1, \ldots, K = 18$.  Since this represents the
variance stabilizing transformation of a binomial distribution, it is
reasonable to suppose that
$$ y_i | \theta_i \sim \mathrm{N}(\theta_i, 1) ~~\text{for}~ i = 1,
\ldots, K \; .
$$ 
Here we consider a hierarchical model proposed by \cite{rose:1996}.
Specifically we further assume that
$$
\theta_1, \ldots, \theta_K  ~\text{are i.i.d.}~ \mathrm{N}(\mu, \lambda)
$$
where
$$
p(\mu, \lambda) \propto \lambda^{-(b+1)} e^{-c/\lambda} I(\lambda > 0) 
$$ with $b$ and $c$ known hyperparameters; thus $\mu$ has the flat
prior and $\lambda$ has an {\em inverse gamma} prior.  This results in
a proper posterior having dimension $K + 2 = 20$.  \cite{rose:1996}
developed a block Gibbs sampler for simulating from the posterior
distribution of $(\theta_1, \ldots, \theta_K, \mu, \lambda)$ and
proved that the resulting Markov chain is geometrically ergodic.
\cite{jone:hara:caff:neat:2006} showed how to implement regenerative
simulation.

Suppose we are interested in estimating the posterior quantiles
$\xi_q^{(i)}$ of a particular $\theta_i$, representing the ``true''
(transformed) batting average of a particular ballplayer.  We conduct
a simulation study to assess the performance of the confidence
intervals at \eqref{eq:ci bm}, \eqref{eq:ci sbm}, and \eqref{eq:ci
  rs}, corresponding to BM, SBM, and RS, respectively.

\cite{jone:hara:caff:neat:2006} showed how to simulate independent draws from the 
posterior distribution via rejection sampling.  Setting hyperparameter values at 
$b = c = 2$, we generated 2E7 iterations of the rejection sampler to estimate the quantiles 
$\xi_q^{(9)}$---the 9th player in \pcite{efro:morr:1975} data set was Ron Santo of the 
Chicago Cubs---and obtained the 
quantiles summarized in Table \ref{tab:baseball_truth}.  
We then ran 5000 replications of \pcite{rose:1996} Gibbs sampler for 50 regenerations each.  
Using the regeneration recipe of \cite{jone:hara:caff:neat:2006}, the mean tour length 
was about 28 updates, with a standard deviation of approximately 28 as well.  For each realized 
chain, we computed 95\% confidence intervals for $\xi_q^{(9)}$ using each of \eqref{eq:ci bm}, 
\eqref{eq:ci sbm}, and \eqref{eq:ci rs}.  
Empirical coverage rates (with the values in Table \ref{tab:baseball_truth} taken as the 
``truth'') are reported in Table 
\ref{tab:baseball_cov}, and interval half-widths are summarized in Table \ref{tab:baseball_width}.  

\begin{table}
\begin{center}
\begin{tabular}{c|ccccc}
$q$~ &~ 0.1 & 0.3 & 0.5 & 0.7 & 0.9 \\
\hline
$\xi_q^{(9)}$~ & ~-4.278 & -3.771 & -3.428 & -3.087 & -2.590 \\
MCSE~ & ~(2.6E-4) & (1.9E-4) & (1.8E-4) & (1.9E-4) & (2.5E-4) \\
\end{tabular}
\caption{Monte Carlo estimates of posterior quantiles for $\theta_9$ in example of Section 
\ref{sub:baseball}, taken as the ``truth'' in subsequent analysis.  Based on 2E7 independent 
draws.  }
\label{tab:baseball_truth}
\end{center}
\end{table}

\begin{table}
\begin{center}
\begin{tabular}{|c|ccccc|}
\hline
 & \multicolumn{5}{c|}{$q$} \\
 Method & 0.1 & 0.3 & 0.5 & 0.7 & 0.9 \\
 \hline
 BM & 0.936 & 0.939 & 0.942 & 0.944 & 0.934 \\
 SBM & 0.941 & 0.937 & 0.939 & 0.940 & 0.941 \\
 RS &  0.932 & 0.938 & 0.940 & 0.940 & 0.931 \\
 \hline
 \end{tabular}
 \caption{Empirical coverage rates of nominal 95\% confidence intervals for $\xi_q^{(9)}$ in 
 example of Section \ref{sub:baseball}.  Based on 5000 simulations, MCSEs range from 
 3.3E-3 to 3.6E-3.    }
 \label{tab:baseball_cov}
 \end{center}
 \end{table}
 
 \begin{table}
 \begin{center}
 \begin{tabular}{|c|ccc|}
 \hline
  & \multicolumn{3}{c|}{Method} \\
$q$ & BM & SBM & RS \\
\hline   
0.1 & 0.0650 (0.010) & 0.0656 (0.008) & 0.0651 (0.011) \\
0.3 & 0.0514 (0.008) & 0.0506 (0.006) & 0.0519 (0.008) \\
0.5 & 0.0490 (0.007) & 0.0479 (0.006) & 0.0494 (0.008) \\
0.7 & 0.0507 (0.007) & 0.0497 (0.006) & 0.0511 (0.008) \\
0.9 & 0.0623 (0.009) & 0.0631 (0.008) & 0.0629 (0.011) \\
\hline
  \end{tabular}
  \caption{Mean (and standard deviation) of CI half-widths for nominal 95\% confidence intervals 
  for $\xi_q^{(9)}$ in example of Section \ref{sub:baseball}, based on 5000 replications.}
  \label{tab:baseball_width}
  \end{center}
  \end{table}

\section{Discussion}
\label{sec:Discussion}

We have focused on assessing the Monte Carlo error for estimating
quantiles in MCMC settings.  In particular, we established quantile
CLTs and considered using batch means, subsampling and regenerative
simulation to estimate the variance of the asymptotic Normal
distributions.  We also studied the finite-sample properties of the
resulting confidence intervals in the context of three examples.

Overall, the finite-sample properties were comparable across the three
variance estimation techniques considered.  However, SBM required
substantially more computational effort because it orders each of the
$n-b+1$ overlapping blocks to obtain the quantile estimates.  For
example, we ran a three dimensional probit regression Markov chain
(Section~\ref{sec:probit}) for $2 \times 10^5$ iterations and
calculated an MCSE for the median of the three marginals.  The BM
calculation took 0.37 seconds while the SBM calculation took 84.04
seconds, or 227 times longer.

The conditions required in the CLT in Theorem~\ref{thm:clt} are the
same as those required in the CLT of Theorem~\ref{thm:clt regen}.
However, RS requires stronger conditions in the sense that it requires
the user to establish a \textit{useful} minorization condition
\eqref{eq:minor.gen}.  Although minorization conditions are often
nearly trivial to establish, they are seen as a substantial barrier by
practitioners because they require a problem-specific approach.
Alternatively, it is straightforward to implement the BM-based and
SBM-based approaches in general software--see the recent \verb@mcmcse@
R package \citep{fleg:hugh:2012} which implements the methods of this
paper.

\begin{appendix}

\section{Preliminaries:  Markov chains as mixing processes}
\label{sec:mixing}

Let $S = \{S_{n}\}$ be a strictly stationary stochastic process on a probability space $(\Omega, {\mathcal F}, P)$ and set ${\mathcal F}_{k}^{l} = \sigma(S_{k}, \ldots, S_{l})$.  Define the $\alpha$-mixing coefficients for $n=1,2,3,\ldots$ as
\[ 
\alpha(n) = \sup_{k \ge 1} \sup_{A \in {\mathcal F}_{1}^{k}, \, B \in {\mathcal F}_{k+n}^{\infty}} | P(A \cap B) - P(A) P(B) | \; .
\]
Let $f : \Omega \to \mathbb{R}$ be Borel.  Set $T = \{ f(S_{n}) \}$
and let $\alpha_{T}$ and $\alpha_{S}$ be the $\alpha$-mixing
coefficients for $T$ and $S$, respectively.  Then by elementary
properties of sigma-algebras \citep[cf.][p. 16]{chow:teic:1978}
$\sigma(T_{k}, \ldots, T_{l}) \subseteq \sigma(S_{k}, \ldots, S_{l}) =
{\mathcal F}_{k}^{l}$ and hence $\alpha_{T}(n) \le \alpha_{S}(n)$ for
all $n$.

Define the $\beta$-mixing coefficients for $n=1,2,3,\ldots$ as
\[
\beta(n) = \underset{B_{1}, \ldots, B_{J} \text{ partition }
  \Omega}{\underset{B_{1}, \ldots, B_{J} \in {\mathcal
   F}_{m+n}^{\infty}}{\underset{A_{1}, \ldots, A_{I} \text{ partition }
      \Omega}{\underset{A_{1}, \ldots, A_{I} \in {\mathcal
          F}_{1}^{m}}{\underset{m \in \mathbb{N}}{\sup}}}}}
\frac{1}{2}\sum_{i=1}^{I} \sum_{j=1}^{J} \left| P(A_{i} \cap B_{j}) -
  P(A_{i}) P(B_{j})\right| \; .
\]
If $\beta(n) \to 0$ as $n \to \infty$, we say that $S$ is
\textit{$\beta$-mixing} while if $\alpha(n) \to 0$ as $n \to \infty$,
we say that $S$ is \textit{$\alpha$-mixing}.  It is easy to prove that
$2 \alpha(n) \le \beta(n)$ \citep[see][for discussion of this and other inequalities]{brad:1986} for all $n$
so that $\beta$-mixing implies $\alpha$-mixing.

Let $X$ be a stationary Harris ergodic Markov chain on $(\sX,
\mathcal{B}(\sX))$, which has invariant distribution $\pi$.  In this
case the expressions for the $\alpha$- and $\beta$-mixing coefficients
can be simplified
\[
\alpha(n) = \sup_{A, \, B \in {\mathcal B}} \left| \int_{A} \pi(dx)
P^{n}(x, B) - \pi(A) \pi(B) \right| 
\]
while 
\citet{davy:1973} showed that
\begin{equation}
\label{eq:beta mixing}
\beta(n) = \int_{\sX} \|P^{n}(x, \cdot) - \pi(\cdot) \| \pi(dx)
\; .
\end{equation}

\begin{theorem}
\label{thm:basic mixing}
A stationary Harris ergodic Markov chain is $\beta$-mixing, hence
$\alpha$-mixing.  In addition, if \eqref{eq:tvn} holds, then $\beta(n)
\le \psi(n) E_{\pi}M$ for all $n$.
\end{theorem} 

\begin{proof}
  The first part is Theorem 4.3 of \citet{brad:1986} while the second
  part can be found in the proof of Theorem 2 in
  \citet{chan:geye:1994}.
\end{proof}
Since $2 \alpha(n) \le \beta(n)$ we observe that
Theorem~\ref{thm:basic mixing} ensures that if $p \ge 0$, then
\begin{equation}
\label{eq:tv to mixing}
\sum_{n=1}^{\infty} n^{p} \psi(n) < \infty \hspace*{4mm} \text{
  implies } \hspace*{4mm} \sum_{n=1}^{\infty} n^{p} \alpha(n) < \infty
\; .
\end{equation}

\section{Proofs}

\subsection{Proof of Proposition~\ref{prop:wangetal}}
\label{app:wangetal}

We begin by showing that we can weaken the conditions of Lemma 3.3 in \citet{wang:hu:yang:2011}.

\begin{lemma} \label{lem:wangetal}
  Let $S=\{S_{n}\}$ be a stationary $\alpha$-mixing process such that
  $\alpha_{S}(n) \le C n^{-\beta}$ for some $\beta > 1$ and positive
  finite constant $C$. Assume the common marginal distribution
  function $F$ is absolutely continuous with continuous density
  function $f$ such that $0 < f(\xi_{q}) < \infty$. For any $\theta>0$
  and $\delta \in (9/(10+8\beta), \, 1/2)$ there exists $n_{0}$ so
  that if $n \ge n_{0}$ then with probability 1
\[
|\hat{\xi}_{n, q} - \xi_{q} | \le \frac{\theta(\log \log n )^{1/2}}{f(\xi_{q}) n^{1/2 - \delta}} \; .
\]
\end{lemma}

\begin{proof}
  Let $\epsilon_{n} = \theta(\log \log n )^{1/2} / f_{V}(\xi_{p})
  n^{1/2 - \delta}$.  Set $\delta_{n1} = F(\xi_{q} + \epsilon_{n}) -
  F(\xi_{q})$ and note that by Taylor's expansion there exists $0 < h
  < 1$ such that
\begin{align*}
  \delta_{n1} & = \epsilon_{n} f(\xi_{q}) \frac{f(h
    \epsilon_{n} + \xi_{q})}{f(\xi_{q})} \; .
\end{align*}
Also, note that 
\[
\frac{f(h \epsilon_{n} + \xi_{q})}{f(\xi_{q})} \to 1 \hspace*{5mm} n \to \infty 
\]
and hence for sufficiently large $n$
\[
\frac{f(h \epsilon_{n} + \xi_{q})}{f(\xi_{q})} \ge \frac{1}{2} \; .
\]
Then for sufficiently large $n$
\[
\delta_{n1} \ge \frac{1}{2} \epsilon_{n} f(\xi_{q})  = \frac{\theta}{2} \frac{(\log \log n )^{1/2}}{n^{1/2 - \delta}} \; .
\]
A similar argument shows that for sufficiently large $n$
\[
\delta_{n2} = F (\xi_{q}) - F(\xi_{q} - \epsilon_{n}) \ge \frac{\theta}{2} \frac{(\log \log n )^{1/2}}{n^{1/2 - \delta}} \; .
\]
The remainder exactly follows the proof of Lemma 3.3 in
\citet{wang:hu:yang:2011} and hence is omitted.
\end{proof}

The proof of Proposition~\ref{prop:wangetal} will follow directly from the following Corollary.

\begin{corollary} \label{cor:wangetal}
  Suppose the stationary Markov chain $X$ is polynomially ergodic of
  order $m > 1$.  For any $\theta>0$ and $\delta \in (9/(10+8m), \,
  1/2)$ with probability 1 for sufficiently large $n$
\[
|\hat{\xi}_{n, q} - \xi_{q} | \le \frac{\theta(\log \log n )^{1/2}}{f_{V}(\xi_{q}) n^{1/2 - \delta}}
\]
and hence there is a positive constant $C_{0}$ such that $\hat{\xi}_{n, q} \in [\,\xi_{q} - C_{0}n^{-1/2 + \delta} \sqrt{\log n}, \, \xi_{q} + C_{0}n^{-1/2 + \delta} \sqrt{\log n}\,]$ with probability 1 for sufficiently large $n$.
\end{corollary}

\begin{proof}
Let $\alpha_{Y}(n)$ be the strong mixing coefficients for $Y=\{g(X_{n})\}$ and note that $\alpha_{Y}(n) \le n^{-m} E_{\pi}M$ by Theorem~\ref{thm:basic mixing}.  The remainder follows from Lemma~\ref{lem:wangetal} and our basic assumptions on $F_{V}$ and $f_{V}$.
\end{proof}

\subsection{Proof of Theorems~\ref{thm:complete bosq} and~\ref{thm:complete}}
\label{app:proof complete} 
We begin with some preliminary results.

\begin{lemma} \label{lem:hoef bosq} Let $X$ be stationary with
  $\beta$--mixing coefficients $\beta(n)$.  Suppose $h : \sX \to
  \mathbb{R}$ and set $W=\{ h(X_{n}) \}$.  If $|| h || :=
  \sup_{x \in \sX} | h(x) | < \infty$, then for any integer $a \in
  \left[ 1, n/2 \right]$ and each $\epsilon > 0$,
\[
\Pr \left( \left| \sum_{i=0}^{n-1} (W_{i} -  E_{\pi} 
    W_{i})  \right| > n \epsilon \right) \le 4 \exp \left\{ -
  \frac{a \epsilon^2}{8 || h ||^2} \right\} + 11 a \left( 1 +
  \frac{4|| h ||}{\epsilon} \right) ^{1/2}  \beta\left( \left\lfloor
    \frac{n}{2a} \right\rfloor \right) \; .
\]
\end{lemma}

\begin{proof}
  This follows easily by combining observations in
  Appendix~\ref{sec:mixing} with Theorem 1.3 from \citet{bosq:1998}.
\end{proof}

\begin{lemma} \citep[Theorem 2,][]{glyn:ormo:2002} \label{lem:hoef}
  Suppose \ \eqref{eq:unif} holds, and $h : \sX \to \mathbb{R}$ with
  $|| h || := \sup_{x \in \sX} | h(x) | < \infty$.  Set $W=\{ h(X_{n})
  \}$ and let $\epsilon > 0$, then for $n > 2 ||h|| n_{0} / (\lambda
  \epsilon)$
\[
\Pr \left( \sum_{i=0}^{n-1} W_{i} - E \left(\sum_{i=0}^{n-1} W_{i}
  \right) \ge n \epsilon \right) \le \exp \left\{ - \frac{\lambda^2 (
    n\epsilon - 2 ||h|| n_{0} / \lambda)^2}{2 n ||h||^2 n_{0}^2}
\right\} \; .
\]
\end{lemma}

\begin{lemma} \label{lem:bounds} Suppose $X_{0} \sim \pi$ and let $g :
  \sX \to \mathbb{R}$ be Borel, $Y = \{ g(X_{n}) \}$ and $\epsilon >
  0$ If $W_{n} = I(Y_{n} > \xi_{q} + \epsilon)$ and $\delta_1 = F_{V}
  ( \xi_{q} + \epsilon ) -q $, then
\begin{equation}
\label{eq:idk1}
\Pr \left( \hat{\xi}_{n,q} > \xi_{q} + \epsilon \right) \le \Pr \left( \left| \sum_{i = 0}^{n-1} (W_{i} - E_{\pi} W_{i}) \right| > n \delta_1 \right)
\end{equation}
while if $V_n = I( Y_n \le \xi_{q} - \epsilon )$ and $\delta_2 = q -
F_{V} ( \xi_{q} - \epsilon )$, then for $0 < \delta < 1$
\begin{equation}
\label{eq:idk2}
\Pr \left( \hat{\xi}_{n,q} < \xi_{q} - \epsilon \right) \le \Pr \left( \left| \sum_{i = 0}^{n-1} (V_i - E_{\pi} V_i)
    \right| > n \delta_2 \delta\right) \; .
\end{equation}
\end{lemma}

\begin{proof}
We compute
\begin{align}
\Pr \left( \hat{\xi}_{n,q} > \xi_{q} + \epsilon \right) &= \Pr \left( F_n ( \hat{\xi}_{n,q} ) > F_n ( \xi_{q} + \epsilon ) \right) \notag \\
&= \Pr \left( q > F_n ( \xi_{q} + \epsilon ) \right) \notag \\
& = \Pr \left( \sum_{i = 0}^{n-1} I( Y_i > \xi_{q} + \epsilon ) > n (1-q) \right) \notag \\
& = \Pr \left( \sum_{i = 0}^{n-1} (W_i - E_{\pi} W_i) > n \delta_1 \right) \notag \\
& \le \Pr \left( \left| \sum_{i = 0}^{n-1} (W_{i} - E_{\pi} W_{i}) \right| > n \delta_1 \right) \notag \; .
\end{align}
Similarly, 
\begin{align}
  \Pr \left( \hat{\xi}_{n,q} < \xi_{q} - \epsilon \right) & \le \Pr \left( F_n ( \hat{\xi}_{n,q} ) \le F_n ( \xi_{q} - \epsilon ) \right) \notag \\
  & \le \Pr \left( q \le F_n ( \xi_{q} - \epsilon ) \right) \notag \\
  & = \Pr \left( \sum_{i = 0}^{n-1} I( Y_i \le \xi_{q} - \epsilon ) \ge n q \right) \notag \\
  & = \Pr \left( \sum_{i = 0}^{n-1} (V_i - E_{\pi} V_i) \ge n \delta_2 \right) \notag \\
  & \le \Pr \left( \left| \sum_{i = 0}^{n-1} (V_i - E_{\pi} V_i)
    \right| > n \delta_2 \delta \right) \notag \; .
\end{align}
\end{proof}

\begin{proof}[Proof of Theorem~\ref{thm:complete bosq}]
Let $\epsilon > 0$.  Then
\[
\Pr \left( \left| \hat{\xi}_{n,q} - \xi_{q} \right| > \epsilon \right) = \Pr \left( \hat{\xi}_{n,q} > \xi_{q} + \epsilon \right) + \Pr \left( \hat{\xi}_{n,q} < \xi_{q} - \epsilon \right) \; .
\]
From Lemmas~\ref{lem:hoef bosq} and~\ref{lem:bounds}, we have for any
integer $a \in \left[ 1, n/2 \right]$,
\[
\Pr \left( \hat{\xi}_{n,q} > \xi_{q} + \epsilon \right) \le 4 \exp
\left\{ - \frac{a \delta_1^2}{8}  \right\} + 11a \left( 1 +
  \frac{4}{\delta_1} \right) ^{1/2} \beta\left(  \left\lfloor \frac{n}{2a} \right\rfloor \right)
\]
and 
\[
\Pr \left( \hat{\xi}_{n,q} < \xi_{q} - \epsilon \right) \le 4 \exp
\left\{ - \frac{a (\delta_2 \delta)^2}{8}  \right\} + 11a \left( 1 +
  \frac{4}{\delta_2 \delta} \right) ^{1/2}  \beta\left(  \left\lfloor \frac{n}{2a} \right\rfloor \right)\; .
\]
Suppose $\gamma= \min \{ \delta_1 , \delta_2 \delta\}$, then
\[
\Pr \left( \left| \hat{\xi}_{n,q} - \xi_{q} \right| > \epsilon \right) \le 8 \exp \left\{ - \frac{ a \gamma^2}{8}  \right\} + 22 a \left( 1 + \frac{4}{\gamma} \right) ^{1/2}  \beta\left(  \left\lfloor \frac{n}{2a} \right\rfloor \right) \; .
\]
Finally note that by Theorem~\ref{thm:basic mixing}
\[
\beta\left(  \left\lfloor \frac{n}{2a} \right\rfloor \right) \le
\psi\left( \left\lfloor \frac{n}{2a} \right\rfloor\right) E_{\pi} M \; .
\]
\end{proof}

\begin{proof}[Proof of Corollary~\ref{cor:unif bounds}]
As in the proof of Theorem~\ref{thm:complete bosq} we have
\[
\Pr \left( \left| \hat{\xi}_{n,q} - \xi_{q} \right| > \epsilon \right) \le 8 \exp \left\{ - \frac{ a \gamma^2}{8}  \right\} + 22 a \left( 1 + \frac{4}{\gamma} \right) ^{1/2}  \beta\left(  \left\lfloor \frac{n}{2a} \right\rfloor \right) \; .
\]
That 
\[
\beta\left(  \left\lfloor \frac{n}{2a} \right\rfloor \right)  \le (1-\lambda)^{\left \lfloor \frac{n}{2an_{0}} \right \rfloor}
\]
follows from \eqref{eq:beta mixing} and that $\|P^{n}(x,\cdot) -
\pi(\cdot)\| \le (1-\lambda)^{\lfloor n/n_{0} \rfloor}$ for all $n$.
\end{proof}

\begin{proof}[Proof of Theorem~\ref{thm:complete}]
First note that
\[
\Pr \left( \left| \hat{\xi}_{n,q} - \xi_{q} \right| > \epsilon \right) = \Pr \left( \hat{\xi}_{n,q} > \xi_{q} + \epsilon \right) + \Pr \left( \hat{\xi}_{n,q} < \xi_{q} - \epsilon \right) \; .
\]
From Lemmas~\ref{lem:hoef} and~\ref{lem:bounds} we have for $n > 2 n_{0} / (\lambda \delta_1)$
\[
\Pr \left( \hat{\xi}_{n,q} > \xi_{q} + \epsilon \right) \le \exp \left\{ - \frac{\lambda^2 ( n\delta_1 - 2 n_{0} / \lambda)^2}{2 n n_{0}^2}  \right\}
\]
and for $n > 2 n_{0} / (\lambda \delta \delta_2)$
\[
\Pr \left( \hat{\xi}_{n,q} < \xi_{q} - \epsilon \right) \le \exp \left\{ - \frac{\lambda^2 ( n \delta \delta_2 - 2 n_{0} / \lambda)^2}{2 n n_{0}^2}  \right\} \; ,
\]
Suppose $\gamma = \min \{ \delta_1 , \delta \delta_2 \}$, then for $n > 2 n_{0} / (\lambda \gamma)$
\[
\Pr \left( \left| \hat{\xi}_{n,q} - \xi_{q} \right| > \epsilon \right) \le 2 \exp \left\{ - \frac{\lambda^2 ( n\gamma - 2 n_{0} / \lambda)^2}{2 n n_{0}^2}  \right\} \; .
\]
\end{proof}

\subsection{Proof for Example~\ref{ex:linchpin}}  \label{app:linchpin}
Let $q(x)$ denote the density of a $t(3)$ distribution, $f_{X}(x)$
the density of a $t(4)$ distribution, $f_{Y|X}(y|x)$ the density of a
$\text{Gamma}(5/2, \, 2 + x^{2}/2)$ distribution and $\pi(x,y)$ the
density at \eqref{eq:da}. Then the Markov chain has Markov transition
density given by
\[
k(x' , y' | x, y) = f_{Y|X}(y' | x') k(x' | x)
\]
where
\[
k(x' | x) \ge q(x') \left\{ 1 \wedge \frac{f_{X}(x') q(x)}{f_{X}(x) q(x')} \right\} = f_{X}(x') \left\{ \frac{q(x)}{f_{X}(x)} \wedge \frac{q(x')}{f_{X}(x')} \right\} \; .
\]
Since for all $x$
\[
\frac{q(x)}{f_{X}(x)} \ge \frac{\sqrt{9375}}{32 \pi}
\]
we have that for all $x,y$
\[
k(x' , y' | x, y) \ge \frac{\sqrt{9375}}{32 \pi} f_{Y|X}(y' | x')f_{X}(x') = \frac{\sqrt{9375}}{32 \pi}\pi(x', y') 
\]
and our claim follows immediately.

\subsection{Proof of Theorem~\ref{thm:clt}} \label{app:clt proof}

We need some notation and few definitions before we begin; for more
background on what follows the reader should consult
\citet{vaar:well:1996}.  A class $\mathcal{T}$ of a set $S$ is said to
pick out a subset $C$ of the set $\{x_1, \ldots, x_n\} \subset S$ if
$C = T \cap \{ x_1, \ldots, x_n \}$ for some $T \subset \mathcal{T}$.
The class $\mathcal{T}$ is said to shatter $\{x_1, \ldots, x_n\}$ if
it picks out all $2^n$ possible subsets.  $\mathcal{T}$ is a V-C class
if there is some $n < \infty$ such that no subset of size $n$ is
shattered by $\mathcal{T}$.  The subgraph of a function $f \colon S
\to \R$ is the set $\{(s,t): 0 \le t \le f(s) \mbox{ or } f(s) \le t
\le 0 \}$. A class of functions $\cF$ is a V-C subgraph class if the
class of its subgraphs is a V-C class of sets (in $S \times \R$).

 Let $\cF$ be a class of functions and define for $f \in \cF$
\begin{equation*}
  \G_n(f) :=
  \inv{\sqrt{n}}
  \sum_{i=0}^{n-1} (f(X_i)-E(f(X_i))) \; .
\end{equation*}
  If, considered as
a process indexed by $\cF$, $\G_n$ converges to a Gaussian limit
process in the space
\[
l_{\infty} (\cF) := \left\{ g : \cF \to \R \, : \, \sup_{f \in \cF}
|g(f)| < \infty \right\} 
\]
equipped with the supremum metric, then we say that $\{X_i\}$ satisfies a
functional CLT.

We begin with a preliminary result.  
\begin{lemma}
  \label{lem:markov-donsker}
  Let $\cF$ be a measurable uniformly bounded V-C subgraph class of
  functions.  If $X$ is stationary and polynomially ergodic of order
  $m > 1$, then there is a Gaussian process $\{G(f)\}_{f \in \cF}$
  which has a version with uniformly bounded and uniformly continuous
  paths with respect to the $L^{2}(\pi)$-norm such that
  \begin{equation}
    \label{eq:markov-donsker}
    \left\{ n^{-1/2} \sum_{i=1}^{n} (f(X_{i}) - E_{\pi} f)\right\}_{f \in \cF} \Longrightarrow \{G(f)\}_{f \in \cF}~~ \text{ in } ~~ l_{\infty}(\cF) \; .
  \end{equation}
Moreover,
\begin{equation} \label{eq:var}
  \var(G(f)) = E[f(X_0) - E(f(X_0))]^2 + 
  2 \sum_{i=1}^\infty E\left[ (f(X_0)-E(f(X_0))) ( f(X_i)-E(f(X_i))) \right] \; .
\end{equation}
\end{lemma}

\begin{proof}
In light of our Theorem~\ref{thm:basic mixing}, \eqref{eq:markov-donsker} follows from Corollary 2.1 in \citet{arco:yu:1994} and \eqref{eq:var} follows from Theorem 0 in \citet{brad:1985}.
\end{proof}

\begin{proof}[Proof of Theorem~\ref{thm:clt}]
Let $\cI = \{ 1_{(-\infty,t]} \}_{t \in \R}$ and set $\cF = \cI \circ
g = \{ 1_{(-\infty,t]} \circ g \}_{t \in \R}$. The class of indicator
functions $\cI = \{ 1_{(-\infty,t]} \}_{t \in \R}$ is a uniformly
  bounded V-C class (see Example~2.6.1, page 135, and Problem~9, page
  151, in \cite{vaar:well:1996}).  By Lemma~2.6.18(vii), page 147, in
  \cite{vaar:well:1996} $\cI \circ g$ is thus also a V-C class.
  Letting $\F_{n}(t) = (1/n) \sum_{i=0}^{n-1} 1_{(-\infty,t]}(Y_i)$
    and using Lemma~\ref{lem:markov-donsker} shows that the empirical
    process $\sqrt{n} (\F_{n} - F_V)$, satisfies
\begin{equation}
  \label{eq:classicaldonsker}
  \sqrt{n}(\F_{n}-F_V) \Longrightarrow \G
\end{equation}
for a Gaussian process $\G$.  Since $F_V$ is continuously
differentiable on $[a,b] = [F_V^{-1}(p) - \epsilon, F_V^{-1}(q) +
  \epsilon]$ for $0 < p < q < 1$ and some $\epsilon > 0$, with
positive derivative $f_V$, it now follows from Theorem~3.9.4 and
Lemma~3.9.23(i) in \cite{vaar:well:1996} that
\begin{equation}
  \label{eq:empirical-quantile-proc-conv}
  \sqrt{n} \left( \F_n^{-1} - F_V^{-1} \right) 
  \Longrightarrow - \frac{\G \circ F_V^{-1} }{f_V \circ F_V^{-1}}, 
  \; \; \; \; 
  \mbox{ in } l^{\infty}[p,q],
\end{equation}
Now, since the class $\cF = \cI \circ g$ is a (uniformly) bounded
class by \eqref{eq:empirical-quantile-proc-conv} we have that the
variance of $\G(y)$ (which corresponds to evaluating $G$ at $f =
1_{(-\infty,y]} \circ g$) is
\begin{equation}
  E\left( 1_{(-\infty,y]}(Y_0) - F_V(y)\right)^2
  + 2 \sum_{i=1}^\infty E\left( (1_{(-\infty,y]}(Y_0) - F_V(y)) (1_{(-\infty,y]}(Y_i) - F_V(y)) \right)
  = \sigma^2(y),
  \label{eq:var-GG}
\end{equation}
where $\sigma^2(y)$ is as defined in \eqref{eq:variance}.  To finish,
we need to evaluate the processes in
\eqref{eq:empirical-quantile-proc-conv} at $q$ so that
\eqref{eq:var-GG} gives us that the variance of $\G(\xi_q)$ is
$\sigma^2(\xi_q)$ and thus the variance of $- \G(\xi_q) / f_V(\xi_q) =
\sigma^2(\xi_q) /f_V^2(\xi_q)$ as desired.

That the same conclusion holds for any initial distribution follows from the same argument as in Theorem 17.1.6 of \citet{meyn:twee:2009}.
\end{proof}

\subsection{Proof of Theorem~\ref{thm:moments}} \label{app:moments}

There exists $\epsilon > 0$ such that $m > 1 + \epsilon + 2/\delta$.
Using~\eqref{eq:tv to mixing} we have that
\[
\sum_{n=1}^{\infty} n^{\epsilon + 2/\delta} \alpha(n) < \infty \; .
\]
\pcite{samu:2004} Proposition 3.1 implies that $E_Q N_1^{2 + \epsilon + 2 / \delta} < \infty$, and \pcite{samu:2004} 
Corollary 3.5 says there exists $2 < p_1 < 2 + \delta$ such that $E_Q (S_1)^{p_1} < \infty$.

\subsection{Proof of Theorem~\ref{thm:clt regen}} \label{app:RS CLT}

We require a preliminary result before proceeding with the
rest of the proof.

\begin{lemma} \label{lem:cont} If $X$ is polynomially ergodic of order
  $m > 1$, then $\Gamma(y)$ is continuous at $\xi_{q}$.
\end{lemma}

\begin{proof}
Denote the limit from the right and left as $\lim_{y \to x^{+}}$ and $\lim_{y \to x^{-}}$, respectively.  From the assumption on $F_{V}$ it is clear that 
\begin{equation} \label{eq:FequalsF}
\lim_{y \to \xi_{q}^{+}} F_{V} (y) = \lim_{y \to \xi_{q}^{-}} F_{V} (y) \; .
\end{equation}
Recall that
\[
S_{1}(y) = \sum_{i=0}^{\tau_{1} - 1} I(Y_{i} \le y) \; .
\]
Let $Z_{1} (y) = S_{1} (y) - F_{V}(y) N_{1}$ and note $E_{Q} \left[ Z_{1} (y) \right] =0$ since \cite{hobe:jone:pres:rose:2002} show
\begin{equation} \label{eq:hjpr}
E_{Q} S_{1} (y) =  F_{V}(y) E_{Q} N_{1} \text{ for all }y \in \R \; .
\end{equation}
Equations \eqref{eq:FequalsF} and \eqref{eq:hjpr} yield $E_{Q} \left[ \lim_{y \to \xi_{q}^{+}} S_{1} (y) \right] = E_{Q} \left[ \lim_{y \to \xi_{q}^{-}} S_{1} (y) \right]$.  The composition limit law and \eqref{eq:FequalsF} result in
\begin{equation}
E_{Q} \left[ \lim_{y \to \xi_{q}^{+}} Z_{1} (y)^2 \right] = E_{Q} \left[ \lim_{y \to \xi_{q}^{-}} Z_{1} (y)^2 \right] . \label{eq:ZequalsZ}
\end{equation}
What remains to show is that the limit of the expectation is the expectation of the limit.  Notice that $0 < S_{1} (y) \le N_{1}$ for all $y \in \R$ and 
\[
\left| Z_{1} (y) \right| = \left| S_{1} (y) - F_{V}(y) N_{1} \right| \le S_{1} (y) + N_{1} \le 2 N_{1} ,
\]
which implies $E_{Q} \left[ Z_{1} (y) ^2 \right] \le 4 E_{Q} N_{1}^2$.
By Theorem~\ref{thm:moments} $E_{Q} N_{1}^2 < \infty$ and the
dominated convergence theorem gives, for any finite $x$,
\begin{align*}
\lim_{y \to x} E_{Q} \left[ Z_{1} (y) ^2 \right] & = E_{Q} \left[ \lim_{y \to x} Z_{1} (y)^2 \right] .
\end{align*}
Finally, from the above fact and \eqref{eq:ZequalsZ} we have $\lim_{y \to \xi_{q}^{+}} E_{Q} \left[ Z_{1} (y) ^2 \right] = \lim_{y \to \xi_{q}^{-}} E_{Q} \left[ Z_{1} (y) ^2 \right]$, and hence $E_{Q} \left[ Z_{1} (y) ^2 \right]$ is continuous at $\xi_{q}$ implying the desired result.
\end{proof}

\cite{hobe:jone:pres:rose:2002} show that $\Gamma(y) = \sigma^{2}(y)
E_{\pi} s $ where $s$ is defined at \eqref{eq:minor.gen}, which yields
the following corollary.

\begin{corollary} \label{cor:cont} 
Under the conditions of Lemma~\ref{lem:cont}, $\sigma^{2}(y)$ is continuous at $\xi_{q}$.
\end{corollary}

\begin{proof}[Proof of Theorem~\ref{thm:clt regen}]
Notice
\begin{align*}
\Pr \left( \sqrt{R} \left( Y_{\tau_{R} (j)} - \xi_{q} \right) \le y \right) & = \Pr \left( Y_{\tau_{R} (j)} \le \xi_{q} + y / \sqrt{R} \right) \\
& = \Pr \left( \sum_{k=0}^{\tau_{R}-1} I \{ Y_{k} \le \xi_{q} + y / \sqrt{R} \} \ge j \right) \\
& = \Pr \left( \sum_{k=0}^{\tau_{R}-1} \left[ I \{ Y_{k} \le \xi_{q} + y / \sqrt{R} \} - F_{V} \left( \xi_{q} + y / \sqrt{R} \right) \right] \right.\\
& \quad \quad \quad \left. \ge j - \tau_{R} F_{V} \left( \xi_{q} + y / \sqrt{R} \right) \right) \\
& = \Pr \left( \frac{\sqrt{R}}{\tau_{R}} \sum_{k=0}^{\tau_{R}-1} W_{R, k} \ge s_{R} \right) \; ,
\end{align*}
where 
\[
W_{R, k} = I \{ Y_{k} \le \xi_{q} + y / \sqrt{R} \} - F_{V} \left( \xi_{q} + y / \sqrt{R} \right), \;~~~ k = 0, \dots, \tau_{R}-1,
\]
and 
\[
s_{R} = \frac{\sqrt{R}}{\tau_{R}} \left(  j - \tau_{R} F_{V} \left( \xi_{q} + y / \sqrt{R} \right) \right) \; .
\]

First, consider the $s_{R}$ sequence.  Let $h : \mathbb{R}^{+} \to
\mathbb{R}^{+}$ satisfy $\lim_{R \to \infty} h \left( \tau_{R} \right)
/ \sqrt{\tau_{R}} = 0$ and set $j=\tau_R q + h(\tau_R)$.  Note that $q
= F_V(\xi_q)$.  For $y \neq 0$
\begin{align*}
  s_R & = \frac{\sqrt{R}}{\tau_R} \left( j - \tau_R F_V(\xi_q + y/\sqrt{R})
  \right) \\
  & = \frac{\sqrt{R}}{\tau_R} \left( \tau_R q + h(\tau_R) - \tau_R F_V(\xi_q
    + y/\sqrt{R})  \right) \\
  & = - \frac{y}{y} \frac{\sqrt{R}}{\tau_R} \left( \tau_R F_V(\xi_q +
    y/\sqrt{R}) - \tau_R q \right)  + \frac{\sqrt{R}}{\tau_R} h(\tau_R)  \\
  & = - y \frac{\sqrt{R}}{y} \left( F_V(\xi_q +
    y/\sqrt{R}) - F_V(\xi_q)  \right)  + \frac{\sqrt{R}}{\tau_R} h(\tau_R) \\
  & = - y  \left(\frac{ F_V(\xi_q +
    y/\sqrt{R}) - F_V(\xi_q)}{ y/\sqrt{R}}  \right)
  +  \frac{h(\tau_R)}{\sqrt{\bar N} \sqrt{\tau_R}},
\end{align*}
which, as $R \to \infty$, converges to $-yf_V(\xi_q)$ since $\bar N \to
E(N_1)$ with probability $1$ where $1 \le E(N_1) < \infty$ by Kac's theorem.  If $y=0$, then $s_{R}= h(\tau_R)/\sqrt{\bar N} \sqrt{\tau_R}$ and hence $s_{R} \to 0$ as $R \to \infty$.  Thus for all $y$ we have $s_{R} \to -yf_V(\xi_q)$ as $R \to \infty$.

Second, consider $W_{R, k}$ 
\begin{equation*}
\frac{\sqrt{R}}{\tau_{R} \left[ \Gamma \left( \xi_{q} + y / \sqrt{R} \right) \right] ^ {1/2}} \sum_{k=0}^{\tau_{R}-1} W_{R, k} \stackrel{d}{\rightarrow} \text{N} (0, 1) \; .
\end{equation*}
Lemma~\ref{lem:cont} and the continuous mapping theorem imply
\begin{equation} \label{eq:cltY}
\frac{\sqrt{R}}{\tau_{R} \left[ \Gamma \left( \xi_{q} \right) \right] ^ {1/2}} \sum_{k=0}^{\tau_{R}-1} W_{R, k} \stackrel{d}{\rightarrow} \text{N} (0, 1) \; .
\end{equation}
Using $s_{R} \to -y f_{V}(\xi_{Q})$ as $R \to \infty$,  \eqref{eq:cltY}, and Slutsky's Theorem, we conclude that, as $R \to \infty$,
\begin{align*}
P \left( \sqrt{R} \left(Y_{\tau_{R} (j)} - \xi_{q} \right) \le y \right) & = P \left( \frac{\sqrt{R}}{\tau_{R} \left[ \Gamma \left( \xi_{q} \right) \right] ^ {1/2}} \sum_{k=0}^{\tau_{R}-1} W_{R, k} \ge \frac{s_{R}}{ \left[ \Gamma \left( \xi_{q} \right) \right] ^ {1/2}} \right) \\
&\rightarrow 1 - \Phi \left\{ \frac{-y f_{V} \left( \xi_{q} \right) }{\left[ \Gamma \left( \xi_{q} \right) \right] ^ {1/2}} \right\}   = \Phi \left\{ \frac{y f_{V} \left( \xi_{q} \right) }{\left[ \Gamma \left( \xi_{q} \right) \right] ^ {1/2}} \right\} \; ,
\end{align*}
resulting in
\[
\sqrt{R} \left( Y_{\tau_{R} (j)} - \xi_{q} \right) \stackrel{d}{\rightarrow} \mathrm{N} \left( 0 , \frac{\Gamma \left( \xi_{q} \right)}{f_{V}^2 \left( \xi_{q} \right)} \right) \; .
\]
\end{proof}

\section{Regenerative simulation in example of Section \ref{sec:t}}
\label{app:regen_details}
The minorization condition necessary for RS is, at least in principle, quite 
straightforward for a Metropolis-Hastings algorithm.  Let $q(x,y)$ denote the 
proposal kernel density, and $\alpha(x,y)$ the acceptance probability.  Then 
$P(x, dy) \geq q(x,y) \alpha(x,y) dy$, since the right hand side only accounts 
for {\it accepted} jump proposals, and the minorization condition is established 
by finding $s'$ and $\nu'$ such that $q(x,y) \alpha(x,y) \geq s'(x) \nu'(y)$.  
By Theorem 2 of \cite{mykl:tier:yu:1995}, the probability of regeneration on 
an {\it accepted} jump from $x$ to $y$ is then given by
$$
r_A(x,y) = \frac{ s'(x) \nu'(y)}{ q(x,y) \alpha(x,y)} \; .
$$
Letting $\pi$ denote the (possibly unnormalized) target density, we have for a 
Metropolis random walk
$$
\alpha(x,y) = \min \left\{ \frac{\pi(y)}{\pi(x)}, ~1 \right\} \geq \min 
\left\{ \frac{c}{\pi(x)}, ~1 \right\} \min \left\{ \frac{\pi(y)}{c}, ~1 \right\} 
$$ 
for any positive constant $c$.  
Further, for any point $\tilde{x}$ and any set $D$ we have
$$
q(x,y) \geq \inf_{y \in D} \left\{ \frac{q(x,y)}{q(\tilde{x},y)} \right\} 
q(\tilde{x}, y) I_D(y) \; .
$$
Together, these inequalities suggest one possible choice of $s'$ and $\nu'$, 
which results in
\begin{equation}
\label{eqn:accept.prob}
r_A(x,y) = I_D(y) \times 
\frac{\inf_{y \in D} \left\{ q(x,y)/q(\tilde{x},y) 
\right\} }{q(x,y) / q(\tilde{x},y)} \times \frac{ \min \left\{ c/\pi(x), 1 \right\} 
\min \left\{ \pi(y)/c , 1 \right\} }{ \min \left\{ \pi(y)/\pi(x), 1 \right\} } \; .
\end{equation}
For a $t(v)$ target distribution, $\alpha(x,y)$ reduces to
$$
\min \left\{ \left( \frac{ v + x^2}{v + y^2} \right)^{\frac{v+1}{2}} 
, ~1 \right\} \geq \min \left\{ \left( \frac{v+x^2}{c} \right)^{\frac{v+1}{2}}, ~ 1 
\right\} \times \min \left\{ \left( \frac{c}{v+y^2} \right)^{\frac{v+1}{2}}, ~ 1 
\right\}
$$
and the last component of \eqref{eqn:accept.prob} is given, up to the constant $c$, 
by
$$
\left[ \frac{ \min \left\{ v + x^2 , c \right\} }{ \min \left\{ v+x^2, v+y^2 
\right\} } \times \frac{ v+y^2}{ \max \left\{ v+y^2, c \right\} } 
\right]^{\frac{v+1}{2}} \; .
$$
Since this piece of the acceptance probability takes the value 1 whenever 
$v + x^2 < c < v + y^2$ or $v + y^2 < c < v + x^2$, it makes sense to take $c$ 
equal to the median value of $v + X^2$ under the target distribution.  

The choice of $\tilde{x}$ and $D$, and the functional form of the middle component 
of \eqref{eqn:accept.prob}, will of course depend on the proposal distribution.  
For the Metropolis random walk with Normally distributed jump proposals, 
$q(x,y) \propto \exp \left\{ -\frac{1}{2 \sigma^2} (y - x)^2 \right\}$, taking 
$D = [\tilde{x} - d, ~\tilde{x} + d]$ for $d > 0$ gives 
$$
\frac{\inf_{y \in D} \left\{ q(x,y)/q(\tilde{x},y) 
\right\} }{q(x,y) / q(\tilde{x},y)} = \exp \left\{ -\frac{1}{\sigma^2} \left\{
(x - \tilde{x})(y - \tilde{x}) + d|x - \tilde{x}| \right\} \right\} \; .
$$
For the $t(v)$ distributions we can take $\tilde{x} = 0$ in all cases, but the 
choice of $d$ should depend on $v$.  With the goal of maximizing regeneration 
frequency, we arrived at, by trial and error, $d = 2\sqrt{ v/(v-2) }$, or two 
standard deviations in the target distribution.

\end{appendix}

\bibliographystyle{apalike} 
\bibliography{ref}

\begin{thebibliography}{}

\bibitem[Acosta et~al., 2014]{acost:hube:jone:2014}
Acosta, F., Huber, M.~L., and Jones, G.~L. (2014).
\newblock Markov chain {M}onte {C}arlo with linchpin variables.
\newblock {\em Preprint}.

\bibitem[Arcones and Yu, 1994]{arco:yu:1994}
Arcones, M.~A. and Yu, B. (1994).
\newblock Central limit theorems for empirical and {$U$}-processes of
  stationary mixing sequences.
\newblock {\em Journal of Theoretical Probability}, 7:47--71.

\bibitem[Baxendale, 2005]{baxe:2005}
Baxendale, P.~H. (2005).
\newblock Renewal theory and computable convergence rates for geometrically
  ergodic {M}arkov chains.
\newblock {\em The Annals of Applied Probability}, 15:700--738.

\bibitem[Bertail and Cl\'{e}men\c{c}on, 2006]{bert:clem:2006}
Bertail, P. and Cl\'{e}men\c{c}on, S. (2006).
\newblock Regenerative block-bootstrap for {M}arkov chains.
\newblock {\em Bernoulli}, 12:689--712.

\bibitem[Bosq, 1998]{bosq:1998}
Bosq, D. (1998).
\newblock {\em Nonparametric Statistics for Stochastic Processes: Estimation
  and Prediction}.
\newblock Springer, New York.

\bibitem[Bradley, 1985]{brad:1985}
Bradley, R.~C. (1985).
\newblock On the central limit question under absolute regularity.
\newblock {\em The Annals of Probability}, 13:1314--1325.

\bibitem[Bradley, 1986]{brad:1986}
Bradley, R.~C. (1986).
\newblock Basic properties of strong mixing conditions.
\newblock In Eberlein, E. and Taqqu, M.~S., editors, {\em Dependence in
  Probability and Statistics: A Survey of Recent Results}, pages 165--192.
  Birkhauser, Cambridge, {MA}.

\bibitem[Brooks and Roberts, 1999]{broo:robe:1999}
Brooks, S.~P. and Roberts, G.~O. (1999).
\newblock On quantile estimation and {M}arkov chain {M}onte {C}arlo
  convergence.
\newblock {\em Biometrika}, 86:710--717.

\bibitem[B\"{u}hlmann, 2002]{buhl:2002}
B\"{u}hlmann, P. (2002).
\newblock Bootstraps for time series.
\newblock {\em Statistical Science}, 17:52--72.

\bibitem[Carlstein, 1986]{carl:1986b}
Carlstein, E. (1986).
\newblock The use of subseries values for estimating the variance of a general
  statistic from a stationary sequence.
\newblock {\em The Annals of Statistics}, 14:1171--1179.

\bibitem[Chan and Geyer, 1994]{chan:geye:1994}
Chan, K.~S. and Geyer, C.~J. (1994).
\newblock Comment on ``{M}arkov chains for exploring posterior distributions''.
\newblock {\em The Annals of Statistics}, 22:1747--1758.

\bibitem[Chow and Teicher, 1978]{chow:teic:1978}
Chow, Y.~S. and Teicher, H. (1978).
\newblock {\em Probability Theory}.
\newblock Springer-Verlag, New York.

\bibitem[Cowles and Carlin, 1996]{cowl:carl:1996}
Cowles, M.~K. and Carlin, B.~P. (1996).
\newblock Markov chain {M}onte {C}arlo convergence diagnostics: {A} comparative
  review.
\newblock {\em Journal of the American Statistical Association}, 91:883--904.

\bibitem[Datta and McCormick, 1993]{datt:mcco:1993}
Datta, S. and McCormick, W.~P. (1993).
\newblock Regeneration-based bootstrap for {M}arkov chains.
\newblock {\em The Canadian Journal of Statistics}, 21:181--193.

\bibitem[Davydov, 1973]{davy:1973}
Davydov, Y.~A. (1973).
\newblock Mixing conditions for {M}arkov chains.
\newblock {\em Theory of Probability and Its Applications}, 27:312--328.

\bibitem[Doss and Tan, 2013]{doss:tan:2013}
Doss, H. and Tan, A. (2013).
\newblock Estimates and standard errors for ratios of normalizing constants
  from multiple markov chains via regeneration.
\newblock {\em Journal of the Royal Statistical Society: Series B {\rm (to
  appear)}}.

\bibitem[Efron and Morris, 1975]{efro:morr:1975}
Efron, B. and Morris, C. (1975).
\newblock Data analysis using {S}tein's estimator and its generalizations.
\newblock {\em Journal of the American Statistical Association}, 70:311--319.

\bibitem[Flegal, 2012]{fleg:2012}
Flegal, J.~M. (2012).
\newblock Applicability of subsampling bootstrap methods in {M}arkov chain
  {M}onte {C}arlo.
\newblock In Wozniakowski, H. and Plaskota, L., editors, {\em Monte Carlo and
  Quasi-Monte Carlo Methods 2010}, volume~23, pages 363--372. Springer
  Proceedings in Mathematics \& Statistics.

\bibitem[Flegal and Gong, 2014]{fleg:gong:2014}
Flegal, J.~M. and Gong, L. (2014).
\newblock Relative fixed-width stopping rules for {M}arkov chain {M}onte
  {C}arlo simulations.
\newblock {\em Statistica Sinica {\rm (to appear)}}.

\bibitem[Flegal et~al., 2008]{fleg:hara:jone:2008}
Flegal, J.~M., Haran, M., and Jones, G.~L. (2008).
\newblock {M}arkov chain {M}onte {C}arlo: Can we trust the third significant
  figure?
\newblock {\em Statistical Science}, 23:250--260.

\bibitem[Flegal and Hughes, 2012]{fleg:hugh:2012}
Flegal, J.~M. and Hughes, J. (2012).
\newblock mcmcse: Monte {C}arlo standard errors for {MCMC} {R} package version
  1.0-1.
\newblock http://cran.r-project.org/web/packages/mcmcse/index.html.

\bibitem[Flegal and Jones, 2010]{fleg:jone:2010}
Flegal, J.~M. and Jones, G.~L. (2010).
\newblock Batch means and spectral variance estimators in {M}arkov chain
  {M}onte {C}arlo.
\newblock {\em The Annals of Statistics}, 38:1034--1070.

\bibitem[Flegal and Jones, 2011]{fleg:jone:2011}
Flegal, J.~M. and Jones, G.~L. (2011).
\newblock Implementing {M}arkov chain {M}onte {C}arlo: Estimating with
  confidence.
\newblock In Brooks, S., Gelman, A., Jones, G., and Meng, X., editors, {\em
  Handbook of {M}arkov Chain {M}onte {C}arlo}, pages 175--197. Chapman \&
  Hall/CRC Press.

\bibitem[Fort and Moulines, 2003]{fort:moul:2003}
Fort, G. and Moulines, E. (2003).
\newblock Polynomial ergodicity of {M}arkov transition kernels.
\newblock {\em Stochastic Processes and their Applications}, 103:57--99.

\bibitem[Geyer, 2011]{geye:2011}
Geyer, C.~J. (2011).
\newblock Introduction to {M}arkov chain {M}onte {C}arlo.
\newblock In {\em Handbook of Markov Chain Monte Carlo}. CRC, London.

\bibitem[Gilks et~al., 1998]{gilk:robe:sahu:1998}
Gilks, W.~R., Roberts, G.~O., and Sahu, S.~K. (1998).
\newblock Adaptive {M}arkov chain {M}onte {C}arlo through regeneration.
\newblock {\em Journal of the American Statistical Association}, 93:1045--1054.

\bibitem[Glynn and Ormoneit, 2002]{glyn:ormo:2002}
Glynn, P. and Ormoneit, D. (2002).
\newblock Hoeffding's inequality for uniformly ergodic {M}arkov chains.
\newblock {\em Statistics \& Probability Letters}, 56:143--146.

\bibitem[Hobert et~al., 2002]{hobe:jone:pres:rose:2002}
Hobert, J.~P., Jones, G.~L., Presnell, B., and Rosenthal, J.~S. (2002).
\newblock On the applicability of regenerative simulation in {M}arkov chain
  {M}onte {C}arlo.
\newblock {\em Biometrika}, 89:731--743.

\bibitem[Hobert et~al., 2006]{hobe:jone:robe:2006}
Hobert, J.~P., Jones, G.~L., and Robert, C.~P. (2006).
\newblock Using a {M}arkov chain to construct a tractable approximation of an
  intractable probability distribution.
\newblock {\em Scandinavian Journal of Statistics}, 33:37--51.

\bibitem[Jarner and Roberts, 2007]{jarn:robe:2007}
Jarner, S.~F. and Roberts, G.~O. (2007).
\newblock Convergence of heavy-tailed {M}onte {C}arlo {M}arkov chain
  algorithms.
\newblock {\em Scandinvian Journal of Statistics}, 24:101--121.

\bibitem[Jones, 2004]{jone:2004}
Jones, G.~L. (2004).
\newblock On the {M}arkov chain central limit theorem.
\newblock {\em Probability Surveys}, 1:299--320.

\bibitem[Jones et~al., 2006]{jone:hara:caff:neat:2006}
Jones, G.~L., Haran, M., Caffo, B.~S., and Neath, R. (2006).
\newblock Fixed-width output analysis for {M}arkov chain {M}onte {C}arlo.
\newblock {\em Journal of the American Statistical Association},
  101:1537--1547.

\bibitem[Jones and Hobert, 2001]{jone:hobe:2001}
Jones, G.~L. and Hobert, J.~P. (2001).
\newblock Honest exploration of intractable probability distributions via
  {M}arkov chain {M}onte {C}arlo.
\newblock {\em Statistical Science}, 16:312--334.

\bibitem[Jones and Hobert, 2004]{jone:hobe:2004}
Jones, G.~L. and Hobert, J.~P. (2004).
\newblock Sufficient burn-in for {G}ibbs samplers for a hierarchical random
  effects model.
\newblock {\em The Annals of Statistics}, 32:784--817.

\bibitem[{\L}atuszy\'{n}ski et~al., 2012]{latu:niem:2012}
{\L}atuszy\'{n}ski, K., Miasojedow, B., and Niemiro, W. (2012).
\newblock Nonasymptotic bounds on the estimation error of {MCMC} algorithms.
\newblock {\em {\rm To appear in} Bernoulli}.

\bibitem[{\L}atuszy\'{n}ski and Niemiro, 2011]{latu:niem:2011}
{\L}atuszy\'{n}ski, K. and Niemiro, W. (2011).
\newblock Rigorous confidence bounds for {MCMC} under a geometric drift
  condition.
\newblock {\em Journal of Complexity}, 27:23--38.

\bibitem[Liu and Wu, 1999]{liu:wu:1999}
Liu, J.~S. and Wu, Y.~N. (1999).
\newblock Parameter expansion for data augmentation.
\newblock {\em Journal of the American Statistical Association}, 94:1264--1274.

\bibitem[Meyn and Tweedie, 2009]{meyn:twee:2009}
Meyn, S. and Tweedie, R. (2009).
\newblock {\em Markov Chains and Stochastic Stability}, volume~2.
\newblock Cambridge University Press Cambridge.

\bibitem[Mykland et~al., 1995]{mykl:tier:yu:1995}
Mykland, P., Tierney, L., and Yu, B. (1995).
\newblock Regeneration in {M}arkov chain samplers.
\newblock {\em Journal of the American Statistical Association}, 90:233--241.

\bibitem[Politis, 2003]{poli:2003}
Politis, D.~N. (2003).
\newblock The impact of bootstrap methods on time series analysis.
\newblock {\em Statistical Science}, 18:219--230.

\bibitem[Politis et~al., 1999]{poli:roma:wolf:1999}
Politis, D.~N., Romano, J.~P., and Wolf, M. (1999).
\newblock {\em Subsampling}.
\newblock Springer-Verlag Inc.

\bibitem[Raftery and Lewis, 1992]{raft:lewi:how:1992}
Raftery, A.~E. and Lewis, S.~M. (1992).
\newblock How many iterations in the {G}ibbs sampler?
\newblock In Bernardo, J.~M., Berger, J.~O., Dawid, A.~P., and Smith, A. F.~M.,
  editors, {\em Bayesian Statistics 4. Proceedings of the Fourth Valencia
  International Meeting}, pages 763--773. Clarendon Press.

\bibitem[Robinson, 1983]{robi:1983}
Robinson, P.~M. (1983).
\newblock Nonparametric estimators for time series.
\newblock {\em Journal of Time Series Analysis}, 4:185--207.

\bibitem[Rosenthal, 1995]{rose:1995a}
Rosenthal, J.~S. (1995).
\newblock Minorization conditions and convergence rates for {M}arkov chain
  {M}onte {C}arlo.
\newblock {\em Journal of the American Statistical Association}, 90:558--566.

\bibitem[Rosenthal, 1996]{rose:1996}
Rosenthal, J.~S. (1996).
\newblock Analysis of the {G}ibbs sampler for a model related to
  {J}ames-{S}tein estimators.
\newblock {\em Statistics and Computing}, 6:269--275.

\bibitem[Roy and Hobert, 2007]{roy:hobe:2007}
Roy, V. and Hobert, J.~P. (2007).
\newblock Convergence rates and asymptotic standard errors for {M}arkov chain
  {M}onte {C}arlo algorithms for {B}ayesian probit regression.
\newblock {\em Journal of the Royal Statistical Society, Series B},
  69:607--623.

\bibitem[Rudolf, 2012]{rudo:2012}
Rudolf, D. (2012).
\newblock Explicit error bounds for {M}arkov chain {M}onte {C}arlo.
\newblock {\em Dissertationes Mathematicae}, 485.

\bibitem[Samur, 2004]{samu:2004}
Samur, J.~D. (2004).
\newblock A regularity condition and a limit theorem for {H}arris ergodic
  {M}arkov chains.
\newblock {\em Stochastic Processes and their Applications}, 111:207--235.

\bibitem[Serfling, 1981]{serf:1981}
Serfling, R.~J. (1981).
\newblock {\em Approximation Theorems of Mathematical Statistics}.
\newblock Wiley-Interscience.

\bibitem[van~der Vaart and Wellner, 1996]{vaar:well:1996}
van~der Vaart, A.~W. and Wellner, J.~A. (1996).
\newblock {\em {Weak Convergence and Empirical Processes}}.
\newblock Springer Series in Statistics. Springer-Verlag, New York.

\bibitem[van Dyk and Meng, 2001]{vand:meng:2001}
van Dyk, D.~A. and Meng, X.-L. (2001).
\newblock The art of data augmentation.
\newblock {\em Journal of Computational and Graphical Statistics}, 10:1--50.

\bibitem[Wang et~al., 2011]{wang:hu:yang:2011}
Wang, X., Hu, S., and Yang, W. (2011).
\newblock The {B}ahadur representation for sample quantiles under strongly
  mixing sequence.
\newblock {\em Journal of Statistical Planning and Inference}, 141:655--662.

\end{thebibliography}

\end{document}